\documentclass{amsart}
\usepackage{amssymb}

\usepackage{amscd}
\usepackage{thmdefs}

\input{tcilatex}
\theoremstyle{definition}
\theoremstyle{remark}
\numberwithin{equation}{section}

\input tcilatex

\begin{document}
\title[Inequalities for infimal convolution]{Integral inequalities for infimal convolution and Hamilton-Jacobi equations}
\author{Patrick J. Rabier}
\address{Department of mathematics, University of Pittsburgh, Pittsburgh, PA 15260}
\email{rabier@imap.pitt.edu}
\subjclass{26D15, 46E30, 35F25, 49L25}
\keywords{Brunn-Minkowski inequality, enclosing ball, Hamilton-Jacobi equations,
infimal convolution, Orlicz space, rearrangement. }
\dedicatory{This paper is dedicated to the memory of Jean Jacques Moreau}
\maketitle

\begin{abstract}
Let $f,g:\Bbb{R}^{N}\rightarrow (-\infty ,\infty ]$ be Borel measurable,
bounded below and such that $\inf f+\inf g\geq 0.$ We prove that with $%
m_{f,g}:=(\inf f-\inf g)/2,$ the inequality $||(f-m_{f,g})^{-1}||_{\phi
}+||(g+m_{f,g})^{-1}||_{\phi }\leq 4||(f\Box g)^{-1}||_{\phi }$ holds in
every Orlicz space $L_{\phi },$ where $f\Box g$ denotes the infimal
convolution of $f$ and $g$ and where $||\cdot ||_{\phi }$ is the Luxemburg
norm (i.e., the $L^{p}$ norm when $L_{\phi }=L^{p}$).

Although no genuine reverse inequality can hold in any generality, we also
prove that such reverse inequalities do exist in the form $||(f\Box
g)^{-1}||_{\phi }\leq 2^{N-1}(||(\check{f}-m_{f,g})^{-1}||_{\phi }+||(\check{
g}+m_{f,g})^{-1}||_{\phi }),$ where $\check{f}$ and $\check{g}$ are suitable
transforms of $f$ and $g$ introduced in the paper and reminiscent of, yet
very different from, nondecreasing rearrangement.

Similar inequalities are proved for other extremal operations and
applications are given to the long-time behavior of the solutions of the
Hamilton-Jacobi and related equations.
\end{abstract}

\section{Introduction\label{intro}}

If $f,g:\Bbb{R}^{N}\rightarrow (-\infty ,\infty ],$ the infimal convolution $%
f\Box g:\Bbb{R}^{N}\rightarrow [-\infty ,\infty ],$ first introduced by
Fenchel \cite{Fe53} and Moreau \cite{Mo63}, \cite{Mo63b}, \cite{Mo70}, is
defined by the formula 
\begin{equation*}
(f\Box g)(x):=\inf_{y\in \Bbb{R}^{N}}(f(x-y)+g(y)).
\end{equation*}
Since then, this operation and its extension to general vector spaces have
found an ever growing variety of applications, including convex functions 
\cite{HiLe96}, \cite{R070}, extension of Lipschitz functions \cite{Hi80},
solutions of the Hamilton-Jacobi equations \cite{AlBaIs99}, \cite{Li82}, 
\cite{St96} and much more (even a proof of the Hahn-Banach theorem \cite
{GlSa10}). In fact, there are by now several thousands publications
using infimal convolution in areas as diverse as image
processing, economics and finance, information theory, probabilities and
statistics, etc. For a glimpse into some of these problems, see the excellent 
recent survey by Lucet  \cite{Lu10}.

In this paper, we investigate the mathematical properties of infimal
convolution in a new direction, by exploring the existence of integral
inequalities involving $f,g$ and $f\Box g.$ The remark that $f\Box g=0$
whenever $f\geq 0$ and $g\geq 0$ are integrable could cast serious doubts on
the value of this program, but they are quickly dispelled by the rebuttal
that no similar triviality arises from the integrability of $f^{-1}$ and $%
g^{-1}.$ Here and everywhere else, $f^{-1}:=1/f,$ $g^{-1}:=1/g,$ etc. This
notation will not be used to denote any set-theoretic inverse.

Omitting technicalities to which we shall return shortly, the first batch of
inequalities will relate the (Luxemburg) norm $||(f\Box g)^{-1}||_{\phi }$
in any Orlicz space $L_{\phi },$ to the norms $||(f-z)^{-1}||_{\phi }$ and $%
||(g+z)^{-1}||_{\phi }$ for a suitable constant $z$ independent of $\phi ,$
to be defined in due time. The only restrictions are that $f$ and $g$ must
be Borel measurable, bounded below and that $f\Box g\geq 0.$ The proofs
depend crucially upon (a slightly weaker form of) the Brunn-Minkowski inequality.

The setting of Orlicz spaces instead of just the classical $L^{p}$ spaces
introduces only mild additional technicalities, is more natural in many
respects and, as we shall see in the examples of Section \ref{application},
is useful in some applications. It does not even require any knowledge of
Orlicz spaces beyond the definitions of Young functions and of the Luxemburg
norm, which will both be reviewed.

This being said, a simple special case asserts that if $f,g\geq 0$ are Borel
measurable and $\inf f=\inf g$ (see Theorem \ref{th7} for a full and much
more general statement) 
\begin{equation}
||f^{-1}||_{p}+||g^{-1}||_{p}\leq 4||(f\Box g)^{-1}||_{p},  \label{1}
\end{equation}
for every $1\leq p\leq \infty ,$ where $||\cdot ||_{p}$ is the norm of $%
L^{p}:=L^{p}(\Bbb{R}^{N}).$ The constant $4$ is best possible among all
constants independent of $p,$ as is readily seen when $f=g=1$ and $p=\infty
. $

The Borel measurability requirement has to do with the measurability of $%
f\Box g,$ without which (\ref{1}) cannot make sense. Curiously, we were
unable to find a discussion of the measurability properties of the infimal
convolution in the classical literature, but the evidence points to the fact
that $f$ and $g$ Lebesgue measurable does not suffice for the measurability
of $f\Box g.$ Indeed, as is well-known, the strict epigraph of $f\Box g$ is
the (vector, also called Minkowski) sum of the strict epigraphs of $f$ and $%
g $ and Sierpi\'{n}ski \cite{Si20} showed, almost a century ago, that the
sum of two Lebesgue measurable sets need not be Lebesgue measurable. In
contrast, the sum of two Borel sets is always Lebesgue measurable (but not
always a Borel set). See Section \ref{background} for further details.

A peculiar feature of (\ref{1}) and of more general similar inequalities is
that only the left-hand side is unchanged by modifications of $f$ and $g$ on
null sets, as long as Borel measurability and $\inf f=\inf g>-\infty $ are
preserved.

Most of the paper is actually devoted to perhaps more important -and
definitely more delicate- reverse inequalities which, in a simpler world,
would read 
\begin{equation}
||(f\Box g)^{-1}||_{p}\leq C(||f^{-1}||_{p}+||g^{-1}||_{p}),  \label{2}
\end{equation}
with $C>0$ independent of $f$ and $g$\ in some suitable class of nonnegative
functions. Unfortunately, the main obstacle to (\ref{2}) is that no remotely
general converse of the Brunn-Minkowski inequality holds in any form, even
for convex sets. Such a converse is actually trivially true for Euclidean
balls, but a direct application of this remark only yields (\ref{2}) for a
narrow subclass of radially symmetric functions.

To take advantage of the converse of the Brunn-Minkowski inequality for
balls in a much broader setting, we introduce a new function transform,
strongly reminiscent of, yet very different from, nonincreasing
rearrangement. The difference is that the upper level sets are rounded
before being rearranged, the rounding being performed by using the concept
of enclosing ball (see Section \ref{transforms}).

To each function $f:\Bbb{R}^{N}\rightarrow [-\infty ,\infty ]$ (no
measurability needed), the aforementioned transform associates a measurable
radially symmetric function $\hat{f},$ which in turn produces another
measurable radially symmetric function $\check{f}:=-(-f\hat{)}.$ In the
special case when $f,g\geq 0$ are Borel measurable and $\inf f=\inf g$ (see
Theorem \ref{th19} for a full and much more general statement), the reverse
inequality in $L^{p}$ reads (compare with (\ref{1}) under the same
assumptions) 
\begin{equation}
||(f\Box g)^{-1}||_{p}\leq 2^{N-1}(||\check{f}^{-1}||_{p}+||\check{g}
^{-1}||_{p}).  \label{3}
\end{equation}

Such an inequality breaks down completely if $\check{f}$ and $\check{g}$ are
replaced with $f$ and $g,$ respectively, even if both functions are radially
symmetric. For example, if $N=1,f(x)=x^{2}+1$ and $g(x)=x^{2}+1$ when $%
x\notin \Bbb{Q},g(x)=1$ if $x\in \Bbb{Q},$ then $f$ and $g$ are Borel
measurable and $\inf f=\inf g=1.$ But $(f\Box g)^{-1}=1/2$ is in no $L^{p}$
space with $p<\infty ,$ whereas $f^{-1}$ and $g^{-1}=f^{-1}$ a.e. are in all
of them. (In this example, it turns out that $\check{f}=f$ but $\check{g}=1.$
) This example also shows that, unlike in (\ref{1}), neither side of (\ref{3}%
) is independent of modifications of $f$ or $g$ on null sets that do not
affect Borel measurability or $\inf f=\inf g.$

When not trivial (i.e., $\check{f}=f$ a.e.), the explicit calculation of $%
\check{f}$ is generally not possible. Nevertheless, the inequality (\ref{3})
is useful because some simple and general conditions about $f$ and $g$
ensure the finiteness of the right-hand side (Lemma \ref{lm21}). There is
certainly more to be discovered in that regard.

\smallskip The proofs of the inequalities involve two other classical
extremal operations 
\begin{equation*}
(f\barwedge g)(x):=\sup_{y}\min \{f(x-y),g(y)\}\text{ and }(f\veebar
g)(x):=\inf_{y}\max \{f(x-y),g(y)\}.
\end{equation*}
Either of these operations fully determines the other, but both notations
will be useful. In general, $f\veebar g=-(-f)\barwedge (-g)$ and, for
nonnegative functions, $f\veebar g=(f^{-1}\barwedge g^{-1})^{-1}$ will be
important. We also prove inequalities similar to (\ref{1}) and (\ref{3}) for
the operations $\barwedge $ and $\veebar $ (both being often referred to as
``level sum'' operations in the literature). In fact, a good part of the
work will consist in proving integral inequalities for $\barwedge ,$ from
which those for $\Box $ and $\veebar $ will be derived.

In the last section, the inequalities are used to obtain $L^{p}$ (and other)
estimates for the inverses of solutions of the Hamilton-Jacobi equations and
variants thereof.

Throughout the paper,\ $\mu _{N}$ denotes the $N$-dimensional Lebesgue
measure and, without a qualifier, measurability always means Lebesgue
measurability.

\section{Background\label{background}}

The purpose of this short section is to review the basic properties of the
operations mentioned in the Introduction, to set the notation used in future
sections and to settle basic measurability issues.

Recall that if $X$ and $Y$ are subsets of $\Bbb{R}^{N},$ their sum $X+Y$ is
defined by 
\begin{equation*}
X+Y:=\left\{ 
\begin{array}{l}
\{x+y:x\in X,y\in Y\}\text{ if }X\neq \emptyset \text{ and }Y\neq \emptyset ,
\\ 
\emptyset \text{ if }X=\emptyset \text{ or }Y=\emptyset .
\end{array}
\right.
\end{equation*}
The following key lemma is well-known. The ``proof'' below merely makes the
connection with the deep property behind it.

\begin{lemma}
\label{lm1} If $X$ and $Y$ are Borel subsets of $\Bbb{R}^{N},$ their sum $%
X+Y $ is measurable\footnote{%
Even a Suslin set, but not necessarily a Borel set.}.
\end{lemma}

\begin{proof}
In Euclidean space (any dimension), the continuous image of a Borel set is
Lebesgue measurable; see Federer \cite[p. 69]{Fe96}. Since $X\times Y$ is a
Borel subset of $\Bbb{R}^{N}\times \Bbb{R}^{N}$ and the addition is
continuous on $\Bbb{R}^{N},$ the result follows.
\end{proof}

Given two functions $f,g:\Bbb{R}^{N}\rightarrow [-\infty ,\infty ]$ and $\xi
\in \Bbb{R},$ call $F_{\xi }^{+}$ and $G_{\xi }^{+}$ the upper level sets 
\begin{equation}
F_{\xi }^{+}:=\{x\in \Bbb{R}^{N}:f(x)>\xi \},\qquad G_{\xi }^{+}:=\{x\in 
\Bbb{R}^{N}:g(x)>\xi \}  \label{4}
\end{equation}
and call $W_{\xi }^{+}$ the corresponding upper level set of $f\barwedge g:$%
\begin{equation}
W_{\xi }^{+}:=\{x\in \Bbb{R}^{N}:(f\barwedge g)(x)>\xi \}.  \label{5}
\end{equation}
It is a standard elementary property that 
\begin{equation}
W_{\xi }^{+}=F_{\xi }^{+}+G_{\xi }^{+}.\text{ }  \label{6}
\end{equation}
By Lemma \ref{lm1} and since $f\veebar g=-(-f)\barwedge (-g),$ it follows at
once from (\ref{6}) that:

\begin{lemma}
\label{lm2}If $f,g:\Bbb{R}^{N}\rightarrow [-\infty ,\infty ]$ are Borel
measurable, then $f\barwedge g$ and $f\veebar g$ are measurable.
\end{lemma}

If $f:\Bbb{R}^{N}\rightarrow [-\infty ,\infty ]$ we set 
\begin{equation}
M_{f}:=\sup f,\qquad m_{f}:=\inf f.  \label{7}
\end{equation}
The next relations are elementary, but important 
\begin{multline}
M_{f\barwedge g}=\min \{M_{f},M_{g}\},\qquad m_{f\barwedge g}\geq \min
\{m_{f},m_{g}\},  \label{8} \\
M_{f\veebar g}\leq \max \{M_{f},M_{g}\},\qquad m_{f\veebar g}=\max
\{m_{f},m_{g}\}.
\end{multline}

We now turn to infimal convolution. Given a function $f:\Bbb{R}%
^{N}\rightarrow (-\infty ,\infty ],$ we denote by $E_{f}:=\{(x,\xi )\in \Bbb{%
R}^{N}\times \Bbb{R}:$ $f(x)<\xi \}$ the strict epigraph of $f.$ It is also
a simple well-known property that if $g:\Bbb{R}^{N}\rightarrow (-\infty
,\infty ]$ is another function, then 
\begin{equation*}
E_{f\Box g}=E_{f}+E_{g}.
\end{equation*}
Since a function is Borel measurable (measurable) if and only if its strict
epigraph is a Borel set (measurable), it follows from Lemma \ref{lm1} that

\begin{lemma}
\label{lm3} If $f,g:\Bbb{R}^{N}\rightarrow (-\infty ,\infty ]$ are Borel
measurable, then $f\Box g$ is measurable.
\end{lemma}

For future use, we also note that if $z\in \Bbb{R},$%
\begin{equation}
f\Box g=(f-z)\Box (g+z).  \label{10}
\end{equation}

\section{First integral inequalities\label{inequalities}}

The Brunn-Minkowski inequality (see e.g. Gardner's survey \cite{Ga02})
asserts that if $X,Y$ are nonempty measurable subsets of $\Bbb{R}^{N}$ \emph{%
and} if $X+Y$ is measurable, then $\mu _{N}(X+Y)^{1/N}\geq \mu
_{N}(X)^{1/N}+\mu _{N}(Y)^{1/N}.$ Obviously, it fails if $X$ or $Y$ is empty
and the other has positive measure. We shall only need the less sharp form 
\begin{equation}
\mu _{N}(X+Y)\geq \mu _{N}(X)+\mu _{N}(Y),\text{ }  \label{11}
\end{equation}
if $X,Y$ and $X+Y$ are measurable and $X\neq \emptyset ,Y\neq \emptyset .$

\begin{lemma}
\label{lm4}If $f,g:\Bbb{R}^{N}\rightarrow [0,\infty ]$ are measurable and $%
M_{f}=M_{g}$ (possibly $\infty ;$ see (\ref{7})) and if $f\barwedge g$ is
measurable, then 
\begin{equation}
\int_{\Bbb{R}^{N}}f+\int_{\Bbb{R}^{N}}g\leq \int_{\Bbb{R}^{N}}f\barwedge g.
\label{12}
\end{equation}
\end{lemma}

\begin{proof}
Set $M:=M_{f}=M_{g}\leq \infty .$ If $M=0,$ then $f=g=$ $f\barwedge g=0$ and
(\ref{12}) is trivial. In what follows, $M>0.$

Since $f\geq 0,$ it is well-known (see (\ref{4})) that $\int_{\Bbb{R}%
^{N}}f=\int_{0}^{\infty }\mu _{N}(F_{\xi }^{+})d\xi .$ By using $F_{\xi
}^{+}=\emptyset $ when $\xi \geq M,$ this reads $\int_{\Bbb{R}%
^{N}}f=\int_{0}^{M}\mu _{N}(F_{\xi }^{+})d\xi .$ Likewise, $\int_{\Bbb{R}%
^{N}}g=\int_{0}^{M}\mu _{N}(G_{\xi }^{+})d\xi ,$ so that $\int_{\Bbb{R}%
^{N}}f+\int_{\Bbb{R}^{N}}g=\int_{0}^{M}(\mu _{N}(F_{\xi }^{+})+\mu
_{N}(G_{\xi }^{+}))d\xi .$

By (\ref{5}) and (\ref{6}) and since $f\barwedge g$ is measurable by
hypothesis, $F_{\xi }^{+}+G_{\xi }^{+}=W_{\xi }^{+}$ is measurable for every 
$\xi .$ If $\xi <M,$ then $F_{\xi }^{+}\neq \emptyset $ and $G_{\xi
}^{+}\neq \emptyset $ by definition of $M$ and so, by (\ref{11}), $\mu
_{N}(F_{\xi }^{+})+\mu _{N}(G_{\xi }^{+})\leq \mu _{N}(F_{\xi }^{+}+G_{\xi
}^{+})=\mu _{N}(W_{\xi }^{+}).$ Therefore, $\int_{\Bbb{R}^{N}}f+\int_{\Bbb{R}
^{N}}g\leq \int_{0}^{M}\mu _{N}(W_{\xi }^{+})d\xi \leq \int_{0}^{\infty }\mu
_{N}(W_{\xi }^{+})d\xi $ (by (\ref{8}), the second inequality is even an
equality). Now, $\int_{0}^{\infty }\mu _{N}(W_{\xi }^{+})d\xi =\int_{\Bbb{R}
^{N}}f\barwedge g$ since $f\barwedge g\geq 0$ and the proof is complete.
\end{proof}

Of course, (\ref{12}) does not follow from a pointwise inequality. The
condition $M_{f}=M_{g}$ cannot be dropped. For example, if $f>0$ and $g=0,$
then $f\barwedge 0=0$ and (\ref{12}) fails.

Lemma \ref{lm4} is just the stepping stone for much more general
inequalities. Recall that in the theory of Orlicz spaces, a nonconstant
function $\phi :[0,\infty ]\rightarrow [0,\infty ]$ is called a \emph{Young
function} if $\phi (0)=0$ and $\phi $ is nondecreasing, convex and left
continuous (\cite{O'Ne65}; see also \cite{AdFo03}, \cite{KrRu61} for a
simplified treatment limited to $N$-functions). In particular, $\phi (\infty
)=\infty .$

\begin{remark}
\label{rm1}If $\phi $ is a Young function and $h:\Bbb{R}^{N}\rightarrow
[0,\infty ]$ is measurable, the monotonicity of $\phi $ shows at once that $%
\phi (h)$ is measurable.
\end{remark}

If $\phi $ is a Young function, the corresponding Orlicz space $L_{\phi }$
consists of all the measurable functions $h$ on $\Bbb{R}^{N}$ such that $%
\int_{\Bbb{R}^{N}}\phi \left( \lambda |h|\right) <\infty $ for some $\lambda
>0$ (this makes sense by Remark \ref{rm1}). It is a (complete) normed space
for the Luxemburg norm $||\cdot ||_{\phi }$ defined by 
\begin{equation}
||h||_{\phi }:=\inf \left\{ r>0:\int_{\Bbb{R}^{N}}\phi (r^{-1}|h|)\leq
1\right\} .  \label{13}
\end{equation}
Since the right-hand side of (\ref{13}) is finite if and only if $h\in
L_{\phi },$ it will always be understood that $||h||_{\phi }=\infty $ when $%
h $ is measurable and $h\notin L_{\phi }.$ Thus, $h\in L_{\phi }$ is
equivalent to $||h||_{\phi }<\infty .$ Furthermore, it is readily checked
that 
\begin{equation}
||h||_{\phi }\leq ||k||_{\phi }\text{ if }|h|\leq |k|  \label{14}
\end{equation}
and, by the left-continuity of $\phi $ and monotone convergence\footnote{%
This is of course a well-known inequality.}, that if $h\in L_{\phi },$%
\begin{equation}
\int_{\Bbb{R}^{N}}\phi (||h||_{\phi }^{-1}|h|)\leq 1.  \label{15}
\end{equation}

If $\phi (\tau ):=\tau ^{p}$ for some $1\leq p<\infty ,$ then $||h||_{\phi
}=||h||_{p}.$ On the other hand, $||h||_{\phi }=||h||_{\infty }$ when $\phi $
is the indicator function of $[0,1]$ ($\phi =0$ in $[0,1]$ and $\infty $
outside).

\begin{lemma}
\label{lm5}If $\phi $ is a Young function and if $h:\Bbb{R}^{N}\rightarrow
[0,\infty ],$ then (see (\ref{7})) $M_{\phi (h)}=\phi (M_{h}).$
\end{lemma}

\begin{proof}
It is plain that $h\leq M_{h}$ implies $\phi (h)\leq \phi (M_{h}),$ so that $%
M_{\phi (h)}\leq \phi (M_{h}).$ It only remains to show that $\phi
(M_{h})\leq M_{\phi (h)},$ which is trivial if $M_{h}=0.$ We henceforth
assume $M_{h}>0.$

By the monotonicity of $\phi $ and $\phi (0)=0,$ there is $\tau _{1}\in
[0,\infty ]$ such that $\phi =\infty $ on $(\tau _{1},\infty ]$ and that $%
\phi <\infty $ on $[0,\tau _{1}).$ Specifically, $\tau _{1}=\sup \{\tau \geq
0:\phi (\tau )<\infty \}.$ If $\tau _{1}\in (0,\infty ),$ then $\phi (\tau
_{1})$ may be finite or infinite. We split the proof into three cases.

(i)\textbf{\ } $M_{h}>\tau _{1}.$ If so, $\tau _{1}<\infty $ and then $\phi
(M_{h})=\infty .$ The set $\{x\in \Bbb{R}^{N}:h(x)>\tau _{1}\}$ is not empty
and $\phi (h(x))=\infty $ for every $x$ in that set. Thus, $\{x\in \Bbb{R}
^{N}:\phi (h(x))=\infty \}\neq \emptyset ,$ so that $M_{\phi (h)}=\infty
=\phi (M_{h}).$

(ii) $M_{h}=\tau _{1}=\infty .$ Then, $\phi (M_{h})=\phi (\infty )=\infty .$
If $\tau >0$ is finite, $\emptyset \neq \{x\in \Bbb{R}^{N}:h(x)>\tau
\}\subset \{x\in \Bbb{R}^{N}:\phi (h(x))\geq \phi (\tau )\}$(by the
monotonicity of $\phi $). As a result, $M_{\phi (h)}\geq \phi (\tau ).$ By
letting $\tau \rightarrow \infty =\tau _{1}$ and since $\lim_{\tau
\rightarrow \infty }\phi (\tau )=\phi (\infty )=\infty $ by the left
continuity of $\phi ,$ it follows that $M_{\phi (h)}=\infty =\phi (M_{h}).$

(iii)\textbf{\ } $0<M_{h}\leq \tau _{1}.$ If $M_{h}=\infty ,$ then $\tau
_{1}=\infty $ and (ii) above applies. Assume now $M_{h}<\infty .$ For every $%
\varepsilon >0,S:=\{x\in \Bbb{R}^{N}:h(x)>M_{h}-\varepsilon \}\neq \emptyset
.$ If $\varepsilon $ is small enough, then $M_{h}-\varepsilon >0$ and $%
S\subset \{x\in \Bbb{R}^{N}:\phi (h(x))\geq \phi (M_{h}-\varepsilon )\}$ by
the monotonicity of $\phi .$ Hence, $M_{\phi (h)}\geq \phi
(M_{h}-\varepsilon ).$ Since $\phi $ is left continuous, $M_{\phi (h)}\geq
\phi (M_{h}).$
\end{proof}

From Lemma \ref{lm4} and Lemma \ref{lm5}, we obtain:

\begin{lemma}
\label{lm6}Let $\phi :[0,\infty ]\rightarrow [0,\infty ]$ be a Young
function. If $f,g:\Bbb{R}^{N}\rightarrow [0,\infty ]$ are Borel measurable
and if $M_{f}=M_{g}$ (possibly $\infty $), then $f\barwedge g$ is measurable
and 
\begin{equation}
\max \{||f||_{\phi },||g||_{\phi }\}\leq ||f\barwedge g||_{\phi }.
\label{16}
\end{equation}
\end{lemma}

\begin{proof}
By Lemma \ref{lm2}, $f\barwedge g$ is measurable and so, by Remark \ref{rm1}%
, $\phi (f),\phi (g)$ and $\phi (f\barwedge g)$ are measurable. Since $\phi
(\min \{f(y),g(x-y)\})=\min \{\phi (f(y)),\phi (g(x-y))\}$ by the
monotonicity of $\phi ,$ we infer that $\sup_{y}\phi (\min
\{f(y),g(x-y)\})=(\phi (f)\barwedge \phi (g))(x).$ By Lemma \ref{lm5} with $%
h(y):=$ $\min \{f(y),g(x-y)\},$ the left-hand side is $\phi ((f\barwedge
g)(x)),$ so that $\phi (f\barwedge g)=\phi (f)\barwedge \phi (g).$ In
particular, $\phi (f)\barwedge \phi (g)$ is measurable.

Since $M_{f}=M_{g},$ then $M_{\phi (f)}=M_{\phi (g)},$ once again by Lemma 
\ref{lm5}. Thus, from the above and from Lemma \ref{lm4} with $f$ and $g$
replaced with $\phi (f)$ and $\phi (g),$ respectively, 
\begin{equation}
\int_{\Bbb{R}^{N}}\phi (f)+\int_{\Bbb{R}^{N}}\phi (g)\leq \int_{\Bbb{R}
^{N}}\phi (f\barwedge g).  \label{17}
\end{equation}

If $r>0,$ then $r^{-1}(f\barwedge g)=r^{-1}f\barwedge r^{-1}g$ and $%
M_{r^{-1}f}=M_{r^{-1}g}.$ Thus, (\ref{17}) for $r^{-1}f$ and $r^{-1}g$
yields $\int_{\Bbb{R}^{N}}\phi (r^{-1}f)\leq \int_{\Bbb{R}^{N}}\phi
(r^{-1}(f\barwedge g)),$ so that $||f||_{\phi }\leq ||f\barwedge g||_{\phi } 
$ by (\ref{13}). Likewise, $||g||_{\phi }\leq ||f\barwedge g||_{\phi }$ and (%
\ref{16}) follows.
\end{proof}

Of course, when $L_{\phi }=L^{1},$ Lemma \ref{lm4} yields the stronger $%
||f||_{1}+||g||_{1}\leq ||f\barwedge g||_{1}$ but (\ref{16}) is optimal when 
$\phi $ is arbitrary (let $f=g=1$ and $L_{\phi }=L_{\infty }$).

We are now in a position to prove our first main integral inequality for
infimal convolution. Recall once more the notation (\ref{7}).

\begin{theorem}
\label{th7} Suppose that $f,g:\Bbb{R}^{N}\rightarrow (-\infty ,\infty ]$ are
Borel measurable, that $m_{f},m_{g}\in \Bbb{R}$ and that $m_{f}+m_{g}\geq 0.$
Set 
\begin{equation}
m_{f,g}=(m_{f}-m_{g})/2.  \label{18}
\end{equation}
\newline
Then, $f-m_{f,g},g+m_{f,g}$ and $f\Box g$ are measurable and nonnegative and 
\begin{equation}
||(f-m_{f,g})^{-1}||_{\phi }+||(g+m_{f,g})^{-1}||_{\phi }\leq 4||(f\Box
g)^{-1}||_{\phi },  \label{19}
\end{equation}
for every Young function $\phi .$
\end{theorem}

\begin{proof}
The measurability of $f\Box g$ was established in Lemma \ref{lm3}. Next, $%
\inf (f-m_{f,g})=\inf (g+m_{f,g})=(m_{f}+m_{g})/2\geq 0,$ whence $f\Box
g=(f-m_{f,g})\Box (g+m_{f,g})\geq 0$ by (\ref{10}). This also implies $%
(f-m_{f,g})^{-1}\geq 0,(g+m_{f,g})^{-1}\geq 0$ and $\sup
(f-m_{f,g})^{-1}=2(m_{f}+m_{g})^{-1}=\sup (g+m_{f,g})^{-1}.$ Therefore, the
inequality (\ref{16}) is applicable in the form 
\begin{equation*}
||(f-m_{f,g})^{-1}||_{\phi }+||(g+m_{f,g})^{-1}||_{\phi }\leq
2||(f-m_{f,g})^{-1}\barwedge (g+m_{f,g})^{-1}||_{\phi }
\end{equation*}
and (\ref{19}) follows from $0\leq h^{-1}\barwedge k^{-1}=(h\veebar
k)^{-1}\leq 2(h\Box k)^{-1}$ when $h$ and $k$ are nonnegative, from $f\Box
g=(f-m_{f,g})\Box (g+m_{f,g})\geq 0$ and from (\ref{14}).
\end{proof}

It was noted in the Introduction that the constant $4$ in (\ref{19}) is
already best possible when $L_{\phi }=L^{\infty }.$

\begin{remark}
Theorem \ref{th7} gives a simple necessary condition for the existence of
solutions of infimal convolution equations (see \cite{Ma91}, \cite{Lu10} and
the references therein): Suppose that $h\geq 0$ is measurable and that $g$
is Borel measurable and bounded below. If $h^{-1}\in L_{\phi }$ for some
Orlicz space $L_{\phi }$ and $||(g-m_{g}+m_{h}/2)^{-1}||_{\phi
}>4||h^{-1}||_{\phi }$ (in particular, if $(g-m_{g}+m_{h}/2)^{-1}\notin
L_{\phi }$), the equation $f\Box g=h$ has no Borel measurable solution $f.$
Indeed, if $f$ exists, then $m_{f}=m_{h}-m_{g}\in \Bbb{R}$ and (\ref{19})
cannot hold.
\end{remark}

If $z\neq 0$ is a constant, there is no simple pointwise relationship
between $(f-z)\veebar (g+z)$ and $f\veebar g.$ As a result, the method of
proof of Theorem \ref{th7} does not yield a variant of (\ref{15}) or (\ref
{19}) with $f\Box g$ replaced with $f\veebar g.$ However, if $f,g\geq 0,$
then $0\leq f\veebar g\leq f\Box g$ and such a variant can be obtained as a
straightforward corollary of Theorem \ref{th7}:

\begin{corollary}
\label{cor8} Suppose that $f,g:\Bbb{R}^{N}\rightarrow [-\infty ,\infty ]$
are Borel measurable, that $g\geq 0$ and that $f\not{\equiv}\infty $ and $g%
\not{\equiv}\infty ,$ so that $0\leq m_{f+}+m_{g}<\infty ,$ where $%
f_{+}:=\max \{f,0\}.$ Then, $f_{+}-m_{f_{+},g},g+m_{f_{+},g}$ (see (\ref{18}
)) and $f\veebar g$ are measurable and nonnegative and 
\begin{equation}
||(f_{+}-m_{f+,g})^{-1}||_{\phi }+||(g+m_{f_{+},g})^{-1}||_{\phi }\leq
4||(f\veebar g)^{-1}||_{\phi },  \label{20}
\end{equation}
for every Young function $\phi .$
\end{corollary}

\begin{proof}
Since $g\geq 0,$ it follows that $0\leq f\veebar g=f_{+}\veebar g\leq $ $%
f_{+}\Box g.$ By Lemma \ref{lm2}, $f\veebar g=f_{+}\veebar g$ is measurable.
Therefore, the corollary follows from (\ref{14}) and from Theorem \ref{th7}
for $f_{+}$ and $g.$
\end{proof}

The constant $4$ is also best possible in (\ref{20}) (among constants
independent of $\phi $): If $f=1$ and $g=\ell >1$ is constant, the
inequality for the $L^{\infty }$ norm is $4(\ell +1)^{-1}\leq 4\ell ^{-1}.$
By letting $\ell \rightarrow \infty ,$ it follows that, in the right-hand
side, $4$ cannot be lowered.

\section{The radial transforms $\hat{f}$ and $\check{f}$\label{transforms}}

The proof of Lemma \ref{lm4} shows that the existence of a converse of the
inequality (\ref{12}), that is, 
\begin{equation*}
\int_{\Bbb{R}^{N}}f\barwedge g\leq C\left( \int_{\Bbb{R}^{N}}f+\int_{\Bbb{R}
^{N}}g\right) ,
\end{equation*}
with $C>0$ independent of $f$ and $g$ would require $\mu _{N}(F_{\xi
}^{+}+G_{\xi }^{+})\leq C(\mu _{N}(F_{\xi }^{+})+\mu _{N}(G_{\xi }^{+}))$
for every $\xi >0.$ However, as pointed out in the Introduction, no converse
of the Brunn-Minkowski inequality or its weaker form (\ref{11}) holds in any
generality.

The transforms defined in this section will enable us (in the next section)
to take advantage of the fact that such a converse trivially exists when $X$
and $Y$ are Euclidean balls. The thought that this case is so special that
it cannot have any broad value would result in a serious oversight.

By a classical theorem of Jung \cite[p. 200]{Fe96}, \cite{Ju01}, every
nonempty bounded subset $X$ of $\Bbb{R}^{N}$ is contained in a \emph{unique}
closed ball $\overline{B}_{X}$ with minimal diameter among all closed balls
containing $X,$ called the \textit{\ enclosing ball} of $X.$ If $X$ is
unbounded, no closed ball contains $X$ and we set $\overline{B}_{X}:=\Bbb{R}%
^{N}.$ Lastly, if $X=\emptyset ,$ every singleton $\{x\}$ satisfies the
``minimal diameter'' requirement, whence uniqueness, but not existence, is
lost. For definiteness, we arbitrarily set $\overline{B}_{\emptyset
}:=\{0\}. $ Evidently, $X\subset \overline{B}_{X}$ in all cases. Jung's
theorem also provides the estimate $\limfunc{diam}(X)\leq \limfunc{diam}(%
\overline{B}_{X})\leq \sqrt{2N/(N+1)}\limfunc{diam}(X),$ but we shall only
make use of the (trivial) first one.

\begin{remark}
\label{rm2}An easily overlooked aspect of enclosing balls is that $X\subset
Y $ implies only $\mu _{N}(\overline{B}_{X})\leq \mu _{N}(\overline{B}_{Y})$
but \emph{not} $\overline{B}_{X}\subset \overline{B}_{Y},$ unless $N=1.$
\end{remark}

The following property of enclosing balls will be important.

\begin{lemma}
\label{lm9}If $X_{n}$ is a nondecreasing sequence of subsets of $\Bbb{R}%
^{N}, $ then $\mu _{N}(\overline{B}_{\cup X_{n}})=\lim \mu _{N}(\overline{B}%
_{X_{n}})=\sup \mu _{N}(\overline{B}_{X_{n}}).$
\end{lemma}

\begin{proof}
Set $X:=\cup X_{n}.$ If $X$ is unbounded, then $\overline{B}_{X}=\Bbb{R}^{N}$
and so $\mu _{N}(\overline{B}_{X})=\infty .$ Since $X:=\cup X_{n}$ and $%
X_{n}\subset X_{n+1},$ the diameter of $X_{n}$ and, hence, that of $%
\overline{B}_{X_{n}},$ tends to $\infty .$ Accordingly, $\lim \mu _{N}(%
\overline{B}_{X_{n}})=\infty .$

Suppose now that $X$ is bounded, so that $\overline{B}_{X}$ is a ball. Since 
$X_{n}\subset X$ implies $\mu _{N}(\overline{B}_{X_{n}})\leq \mu _{N}(%
\overline{B}_{X})$ and since $\mu _{N}(\overline{B}_{X_{n}})$ is
nondecreasing, it is plain that $\lim \mu _{N}(\overline{B}_{X_{n}})\leq \mu
_{N}(\overline{B}_{X}).$ To prove the converse, call $r_{n}\geq 0$ the
radius of $\overline{B}_{X_{n}}.$ The sequence $r_{n}$ is nondecreasing and
bounded above (by the radius of $\overline{B}_{X}$) and so it has a limit $%
r\geq r_{n}$ for every $n.$ As a result, $\lim \mu _{N}(\overline{B}
_{X_{n}}) $ is the measure of any ball with radius $r.$

Next, call $x_{n}$ the center of $\overline{B}_{X_{n}}.$ By a simple
contradiction argument, the sequence $x_{n}$ is bounded (since $x_{n}$ might
not be in $\overline{B}_{X}$ -see Remark \ref{rm2}- this is not totally
trivial). After extracting a subsequence, assume that $x_{n}\rightarrow x\in 
\Bbb{R}^{N}.$ Every $y\in X$ is in $X_{n}$ for $n$ large enough. Since $%
\overline{B}_{X_{n}}\subset \overline{B}(x_{n},r),$ it follows that $y\in 
\overline{B}(x,r).$ Thus, $X\subset \overline{B}(x,r),$ whence $\mu _{N}(%
\overline{B}(x,r))\geq \mu _{N}(\overline{B}_{X})$ by definition of $%
\overline{B}_{X}.$ Since $r=\lim r_{n}$ amounts to $\mu _{N}(\overline{B}%
(x,r))=\lim \mu _{N}(\overline{B}_{X_{n}}),$ it follows that $\lim \mu _{N}(%
\overline{B}_{X_{n}})\geq \mu _{N}(\overline{B}_{X}).$
\end{proof}

Given any function $f:\Bbb{R}^{N}\rightarrow [-\infty ,\infty ],$ we now
proceed to constructing a measurable radially symmetric function $\hat{f}:%
\Bbb{R}^{N}\rightarrow [-\infty ,\infty ]$ whose upper level sets $\hat{F}
_{\xi }^{+}$ have measure equal to $\mu _{N}(\overline{B}_{F_{\xi }^{+}})$
for every $\xi .$ The construction follows that of the nonincreasing
rearrangement of $f.$

\begin{lemma}
\label{lm10}The function $\mu _{N}(\overline{B}_{F_{\xi }^{+}})$ is
nonincreasing and right-continuous on $[-\infty ,\infty ].$
\end{lemma}

\begin{proof}
If $\xi <\eta ,$ then $F_{\eta }^{+}\subset F_{\xi }^{+},$ so that $\mu _{N}(%
\overline{B}_{F_{\eta }^{+}})\leq \mu _{N}(\overline{B}_{F_{\xi }^{+}}).$
For the right continuity, let $\xi _{n}\searrow \xi ,$ so that $F_{\xi
}^{+}=\cup F_{\xi _{n}}^{+}$ and then, by Lemma \ref{lm9}, $\mu _{N}(%
\overline{B}_{F_{\xi }^{+}})=\lim \mu _{N}(\overline{B}_{F_{\xi _{n}}^{+}}).$
\end{proof}

Call $\rho _{f}^{+}(\xi )$ the radius of $\overline{B}_{F_{\xi }^{+}}.$
Since $\rho _{f}^{+}(\xi )$ is proportional to $\mu _{N}(\overline{B}
_{F_{\xi }^{+}})^{1/N},$ it follows from Lemma \ref{lm10} that $\rho
_{f}^{+} $ is nonincreasing and right-continuous. Therefore, 
\begin{equation}
\gamma _{f}^{+}(t):=\inf \{\xi :\rho _{f}^{+}(\xi )\leq t\},  \label{21}
\end{equation}
is a nonincreasing and right-continuous function on $[0,\infty )$ and 
\begin{equation}
\{t\geq 0:\gamma _{f}^{+}(t)>\xi \}=[0,\rho _{f}^{+}(\xi )).  \label{22}
\end{equation}
Indeed, when $f\geq 0$ and $\rho _{f}^{+}(\xi )$ is replaced with $\mu
_{N}(F_{\xi }^{+}),\gamma _{f}^{+}$ becomes the nonincreasing rearrangement
of $f$ and these properties follow uniquely from the monotonicity and
right-continuity of $\mu _{N}(F_{\xi }^{+});$ see for instance 
\cite[pp. 26-27]{Zi89}. We also point out that in most modern expositions,
the nonincreasing rearrangement of a function $f$ is defined to be that of $%
|f|.$ This has not always been the case (see Day \cite{Da72} or Luxemburg 
\cite{Lu67}) and the monotonicity and right-continuity properties of
nonincreasing rearrangements are independent of whether $f$ or $|f|$ is used
in their definition.

We now set 
\begin{equation}
\hat{f}(x):=\gamma _{f}^{+}(|x|).  \label{23}
\end{equation}
Some basic properties of $\hat{f}$ are summarized in the next theorem.

\begin{theorem}
\label{th11}Given $f:\Bbb{R}^{N}\rightarrow [-\infty ,\infty ],$ the
function $\hat{f}$ has the following properties:\newline
(i) $\hat{f}$ is measurable and $\hat{f}=f$ a.e. if and only if $f(x)$ is
a.e. equal to a nonincreasing function of $|x|$. If also $f(x)$ is a
right-continuous function of $|x|,$ then $\hat{f}=f.$ \newline
(ii) $\mu _{N}(\hat{F}_{\xi }^{+})=\mu _{N}(\overline{B}_{F_{\xi }^{+}})$
for every $\xi \in [-\infty ,\infty ],$ where $\hat{F}_{\xi }^{+}$ denotes
the upper $\xi $-level set of $\hat{f}.$\newline
(iii) $M_{\hat{f}}\leq M_{f}$ and $m_{\hat{f}}\geq m_{f}$ (in particular, $%
f\geq 0\Rightarrow \hat{f}\geq 0$). Furthermore, $M_{\hat{f}}=\limfunc{ess}%
\sup \hat{f}$ and $m_{\hat{f}}=\limfunc{ess}\inf \hat{f}.\newline
$(iv) $(f+z\hat{)}=\hat{f}+z$ for $z\in \Bbb{R}$ and $(f(c\cdot )\hat{)}=%
\hat{f}(c\cdot )$ for $c\in \Bbb{R}\backslash \{0\}.$ \newline
(v) $(cf\hat{)}=c\hat{f}$ for every $c\geq 0.$\newline
(vi) If $h:\Bbb{R}^{N}\rightarrow [-\infty ,\infty ]$ and $h\leq f,$ then $%
\hat{h}\leq \hat{f}.$ \newline
(vii) If $f$ is bounded below on bounded subsets and $\lim_{|x|\rightarrow
\infty }f(x)=-\infty $ and if $h:\Bbb{R}^{N}\rightarrow [-\infty ,\infty ]$
satisfies $h(x)\leq f(x)$ for $|x|$ large enough, then $\hat{h}(x)\leq \hat{f%
}(x)$ for $|x|$ large enough. Furthermore, if $f(x)$ is a strictly
decreasing function of $|x|$ and if $h(x)\leq f(x)$ when $x\notin B$ for
some open ball $B$ centered at the origin, then $\hat{h}(x)\leq \hat{f}(x)$ (%
$=f(x)$ by (i)) for every $x\notin B.$
\end{theorem}

\begin{proof}
(i) The measurability of $\hat{f}$ follows at once from the monotonicity of $%
\gamma _{f}^{+}$ and the necessity of the given conditions for $\hat{f}=f$
a.e. is obvious. Conversely, if $f(x)=\gamma (|x|)$ with $\gamma :[0,\infty
)\rightarrow [-\infty ,\infty ]$ nonincreasing, the upper level sets of $f$
are balls centered at the origin (possibly $\Bbb{R}^{N}$) and $\rho _{f}^{+}$
is the distribution function of $\gamma ,$ so that $\gamma _{f}^{+}=\gamma
^{*},$ the nonincreasing rearrangement of $\gamma .$ Since $\gamma
^{*}=\gamma $ except perhaps at the countably many points of discontinuity
of $\gamma ,$ it follows that $\hat{f}=f$ a.e. Clearly, this remains true if 
$f(x)=\gamma (|x|)$ a.e. If $\gamma $ is right-continuous, then $\gamma
^{*}=\gamma $ and $\hat{f}=f.$

(ii) Just notice that, by (\ref{22}) and (\ref{23}), $\hat{F}_{\xi }^{+}$ is
the open ball with center $0$ and radius $\rho _{f}^{+}(\xi ).$

(iii) With no loss of generality, assume $M_{f}<\infty .$ Since $\gamma
_{f}^{+}$ is nonincreasing, $\max \gamma _{f}^{+}=\gamma _{f}^{+}(0)=\inf
\{\xi :\rho _{f}^{+}(\xi )=0\}.$ If $\xi \geq M_{f},$ then $F_{\xi
}^{+}=\emptyset ,$ whence $\overline{B}_{F_{\xi }^{+}}=\{0\}$ and so $\rho
_{f}^{+}(\xi )=0.$ Thus, $\max \gamma _{f}^{+}=\inf \{\xi :\mu _{N}(F_{\xi
}^{+})=0\}\leq M_{f}.$ On the other hand, by (\ref{23}), $\max \gamma
_{f}^{+}=M_{\hat{f}}.$ This shows that $M_{\hat{f}}\leq M_{f}.$ Furthermore, 
$M_{\hat{f}}=\limfunc{ess}\sup \hat{f}$ by (\ref{23}) and the
right-continuity of $\gamma _{f}^{+}.$

That $m_{\hat{f}}\geq m_{f}$ is obvious if $m_{f}=-\infty ,$ or if $%
m_{f}=\infty $ (for then $f=\hat{f}=\infty $). If $m_{f}\in \Bbb{R}$ and $%
\xi <m_{f},$ then $\rho _{f}^{+}(\xi )=\infty .$ Thus, $\gamma _{f}^{+}\geq
m_{f} $ by (\ref{21}) and so $\hat{f}\geq m_{f},$ whence $m_{\hat{f}}\geq
m_{f}.$ That $m_{\hat{f}}=\limfunc{ess}\inf \hat{f}$ follows from the
monotonicity of $\gamma _{f}^{+}.$

(iv) Since $\{x\in \Bbb{R}^{N}:f(x)+z>\xi \}=F_{\xi -z}^{+},$ it follows
that $\rho _{f+z}^{+}(\xi )=\rho _{f}^{+}(\xi -z),$ which in turn yields $%
\gamma _{f+z}^{+}=\gamma _{f}^{+}+z,$ i.e., $(f+z\hat{)}=\hat{f}+z.$ The
proofs that $f(c\cdot \hat{)}=\hat{f}(c\cdot )$ if $c\in \Bbb{R}\backslash
\{0\}$ is equally straightforward.

(v) Since this is trivial when $c=0,$ assume $c>0.$ Then, $\rho
_{cf}^{+}(\xi )=\rho _{f}^{+}(\xi /c),$ whence $\gamma _{cf}^{+}=c\gamma
_{f}^{+},$ i.e. $(cf\hat{)}=c\hat{f}.$

(vi) If $h\leq f,$ then (with a self-explanatory notation) $H_{\xi
}^{+}\subset F_{\xi }^{+}$ and so $\mu _{N}(\overline{B}_{H_{\xi }^{+}})\leq
\mu _{N}(\overline{B}_{F_{\xi }^{+}}).$ Hence, $\rho _{h}^{+}(\xi )\leq \rho
_{f}^{+}(\xi )$ and, by (\ref{21}), $\gamma _{h}^{+}\leq \gamma _{f}^{+},$
so that $\hat{h}\leq \hat{f}.$

(vii) Choose an open ball $B$ centered at the origin such that $h(x)\leq
f(x) $ when $x\notin B.$ Since $f$ is bounded below on bounded subsets, $%
\inf_{B}f $ is finite and, if $\xi _{0}<\inf_{B}f,$ then $B\subset F_{\xi
}^{+}$ for every $\xi \leq \xi _{0}.$ By (vi), $\hat{h}$ is increased when $%
h $ is increased. Thus, if it can be shown that $\hat{h}\leq \hat{f}$ after
increasing $h,$ this inequality also holds before $h$ is increased. In
particular, we may increase $h$ on $B$ so that $\xi _{0}<\inf_{B}h$ and then 
$B\subset H_{\xi }^{+}$ for every $\xi \leq \xi _{0}.$ Thus, $B\subset
H_{\xi }^{+}\cap F_{\xi }^{+}$ for every $\xi \leq \xi _{0}.$ On the other
hand, $\{x\notin B:h(x)>\xi \}\subset \{x\notin B:f(x)>\xi \}$ since $h\leq
f $ on $\Bbb{R}^{N}\backslash B.$ Altogether, if $\xi \leq \xi _{0},$ then $%
H_{\xi }^{+}\subset F_{\xi }^{+}$ and so $\rho _{h}^{+}(\xi )\leq \rho
_{f}^{+}(\xi ).$ As a result, 
\begin{multline}
\gamma _{h}^{+}(t):=\inf \{\xi :\rho _{h}^{+}(\xi )\leq t\}\leq \inf \{\xi
\leq \xi _{0}:\rho _{h}^{+}(\xi )\leq t\}\leq  \label{24} \\
\inf \{\xi \leq \xi _{0}:\rho _{f}^{+}(\xi )\leq t\}.
\end{multline}

Since $\lim_{|x|\rightarrow \infty }f(x)=-\infty ,$ the level set $F_{\xi
_{0}}^{+}$ is bounded, whence $\rho _{f}^{+}(\xi _{0})<\infty .$ Choose any $%
t\geq \rho _{f}^{+}(\xi _{0}).$ If $\xi >\xi _{0},$ then $\rho _{f}^{+}(\xi
)\leq \rho _{f}^{+}(\xi _{0})\leq t$ by the monotonicity of $\rho _{f}^{+},$
so that $\gamma _{f}^{+}(t):=\inf \{\xi :\rho _{f}^{+}(\xi )\leq t\}=\inf
\{\xi \leq \xi _{0}:\rho _{f}^{+}(\xi )\leq t\}.$ By (\ref{24}), $\gamma
_{h}^{+}(t)\leq \gamma _{f}^{+}(t).$ Since this is true for every $t\geq
\rho _{f}^{+}(\xi _{0}),$ it follows that $\hat{h}(x)\leq \hat{f}(x)$ when $%
|x|\geq \rho _{f}^{+}(\xi _{0}).$

To complete the proof, assume in addition that $f(x)$ is a strictly
decreasing function of $|x|.$ We show that $\hat{h}(x)\leq \hat{f}(x)$ when $%
x\notin B.$ If $B=\emptyset ,$ the result follows from (v). From now on,
assume $B\neq \emptyset $ (hence $B\neq \{0\}$ as well since $B$ is open).

By the monotonicity of $f$ in $|x|,$ $\inf_{tB}f<\inf_{B}f$ if $t>1.$ It
follows that $F_{\xi _{0}}^{+}\subset tB$ if $t>1$ and $\xi _{0}$ above is
close enough to $\inf_{B}f.$ Thus, $\overline{B}_{F_{\xi _{0}}^{+}}$ $%
\subset t\overline{B},$ so that $\rho _{f}^{+}(\xi _{0})\leq t\rho $ where $%
\rho $ is the radius of $B.$ From the above, $\hat{h}(x)\leq \hat{f}(x)$
when $|x|\geq t\rho $ and, hence, when $|x|\geq \rho $ by first letting $%
t\rightarrow 1$ (which gives only $|x|>\rho $) and next using the
right-continuity of $\hat{f}$ and $\hat{h}$ with respect to $|x|.$ Since $%
\rho $ is the radius of $B$ and $B$ is centered at the origin, this means
that $\hat{h}(x)\leq \hat{f}(x)$ when $x\notin B.$
\end{proof}

Even when $f\geq 0,$ it is not true that $\hat{f}=0$ implies $f=0.$ The
following characterization is important for the proof of the reverse
inequalities (specifically, of Theorem \ref{th18} later).

\begin{lemma}
\label{lm12} If $f:\Bbb{R}^{N}\rightarrow [0,\infty ]$ and $\hat{f}=0,$ then
either $f=0$ or there are $x_{0}\in \Bbb{R}^{N}$ and $0<z\leq \infty $ such
that $f(x)=0$ if $x\neq x_{0}$ and $f(x_{0})=z.$
\end{lemma}

\begin{proof}
By (\ref{21}) and (\ref{23}) and since $\rho _{f}^{+}$ is nonincreasing, $%
\hat{f}=0$ means $\rho _{f}^{+}(0)=0,$ i.e., that the enclosing ball $%
\overline{B}_{F_{0}^{+}}$ has radius $0,$ which happens only when $%
F_{0}^{+}=\emptyset $ or when $F_{0}^{+}=\{x_{0}\}$ is a singleton. If $%
F_{0}^{+}=\emptyset ,$ then $f=0$ since $f\geq 0.$ Suppose now that $%
F_{0}^{+}=\{x_{0}\}.$ This means that $x_{0}$ is the only point where $f$ is
positive, so that $f(x_{0})=z>0$ (possibly $\infty $) and that $f(x)\leq 0$
when $x\neq x_{0}.$ Since $f\geq 0,$ it follows that $f(x)=0$ when $x\neq
x_{0}.$
\end{proof}

\begin{remark}
\label{rm3}By the argument of the above proof, the inequality $M_{\hat{f}%
}\leq M_{f}$ in Theorem \ref{th11} (iii) can be made precise: $M_{\hat{f}%
}<M_{f}$ if and only if some upper level set of $f$ is a singleton. Put
differently, $M_{\hat{f}}=M_{f}$ if and only if there is a sequence of \emph{%
distinct} points $x_{n}$ such that $f(x_{n})\rightarrow M_{f}.$
\end{remark}

We shall also need the transform defined by 
\begin{equation}
\check{f}:=-(-f\hat{)},  \label{25}
\end{equation}
directly given by the formula 
\begin{equation}
\check{f}(x)=\gamma _{f}^{-}(|x|),  \label{26}
\end{equation}
where 
\begin{equation}
\gamma _{f}^{-}(t):=\sup \{\xi :\rho _{f}^{-}(\xi )\leq t\}  \label{27}
\end{equation}
and $\rho _{f}^{-}(\xi )$ is the radius of $\overline{B}_{F_{\xi }^{-}}$
with, of course, $F_{\xi }^{-}=\{x\in \Bbb{R}^{N}:f(x)<\xi \}.$

\begin{theorem}
\label{th13}Given $f:\Bbb{R}^{N}\rightarrow [-\infty ,\infty ],$ the
function $\check{f}$ has the following properties:\newline
(i) $\check{f}$ is measurable and $\check{f}=f$ a.e. if and only if $f(x)$
is a.e. equal to a nondecreasing function of $|x|$. If also $f(x)$ is a
right-continuous function of $|x|,$ then $\check{f}=f$.\newline
(ii) $\mu _{N}(\check{F}_{\xi }^{-})=\mu _{N}(\overline{B}_{F_{\xi }^{-}})$
for every $\xi \in [-\infty ,\infty ],$ where $\check{F}_{\xi }^{-}$ denotes
the lower $\xi $-level set of $\check{f}.$\newline
(iii) $M_{\check{f}}\leq M_{f}$ and $m_{\check{f}}\geq m_{f}$ (in
particular, $f\geq 0\Rightarrow \check{f}\geq 0$). Furthermore, $M_{\check{f}%
}=\limfunc{ess}\sup \check{f}$ and $m_{\check{f}}=\limfunc{ess}\inf \check{f}%
.\newline
$(iv) $(f+z\check{)}=\check{f}+z$ for every $z\in \Bbb{R}$ and $(f(c\cdot )%
\check{)}=\check{f}(c\cdot )$ for $c\in \Bbb{R}\backslash \{0\}.$\newline
(v) $(cf\check{)}=c\check{f}$ for every $c\geq 0$ (and $(cf\check{)}=c\hat{f}
$ for every $c<0$).\newline
(vi) If $h:\Bbb{R}^{N}\rightarrow [-\infty ,\infty ]$ and $h\leq f,$ then $%
\check{h}\leq \check{f}.$\newline
(vii) If $h:\Bbb{R}^{N}\rightarrow [-\infty ,\infty ]$ is bounded above on
bounded subsets, $\lim_{|x|\rightarrow \infty }h(x)=\infty $ and $h(x)\leq
f(x)$ for $|x|$ large enough, then $\check{h}(x)\leq \check{f}(x)$ for $|x|$
large enough. Furthermore, if $h(x)$ is a strictly increasing function of $%
|x|$ and if $h(x)\leq f(x)$ when $x\notin B$ for some open ball $B$ centered
at the origin, then $\check{h}(x)$ ($=h(x)$ by (i)) $\leq \check{f}(x)$ for
every $x\notin B.$\newline
(viii) If $f\geq 0,$ then $(f^{-1}\hat{)}=(\check{f})^{-1}.$\newline
(ix) $(f_{+}\check{)}=\check{f}_{+}.$
\end{theorem}

\begin{proof}
Parts (i) to (vii) follow at once from (\ref{25}) and from the corresponding
properties in Theorem \ref{th11}.

(viii) First, note that if $f\geq 0$ and $\xi <0,$ then $\rho _{f}^{+}(\xi
)=\infty $ and $\rho _{f}^{-}(\xi )=0.$ Hence, by (\ref{21}) and (\ref{27}), 
$\gamma _{f}^{+}(t):=\inf \{\xi \geq 0:\rho _{f}^{+}(\xi )\leq t\}$ and $%
\gamma _{f}^{-}(t):=\sup \{\xi \geq 0:\rho _{f}^{-}(\xi )\leq t\}\geq 0.$
Upon replacing $f$ with $f^{-1}$ in the former formula, we get (since
passing from $f$ to $f^{-1}$ changes upper level sets into lower ones) $%
\gamma _{f^{-1}}^{+}(t):=\inf \{\xi \geq 0:\rho _{f}^{-}(\xi ^{-1})\leq t\}.$
On the other hand, since $\inf (S^{-1})=(\sup S)^{-1}$ for every subset $%
S\subset [0,\infty ],$ we have $(\gamma _{f}^{-}(t))^{-1}=\inf \{\xi
^{-1}\geq 0:\rho _{f}^{-}(\xi )\leq t\}=\inf \{\xi \geq 0:\rho _{f}^{-}(\xi
^{-1})\leq t\}.$ This shows that $\gamma _{f^{-1}}^{+}=(\gamma
_{f}^{-})^{-1} $ and the result follows from (\ref{23}) and (\ref{26}).

(ix) The lower level sets $\{x\in \Bbb{R}^{N}:f_{+}(x)<\xi \}$ are empty if $%
\xi \leq 0$ and coincide with $F_{\xi }^{-}$ if $\xi >0.$ Thus, $\rho
_{f_{+}}^{-}(\xi )=0$ if $\xi \leq 0$ and $\rho _{f_{+}}^{-}(\xi )=\rho
_{f}^{-}(\xi )$ if $\xi >0.$ Therefore, by (\ref{27}), 
\begin{multline}
\gamma _{f_{+}}^{-}(t)=\max \{\sup \{\xi \leq 0:\rho _{f_{+}}^{-}(\xi )\leq
t\},\sup \{\xi >0:\rho _{f_{+}}^{-}(\xi )\leq t\}\}=  \label{28} \\
\max \{0,\sup \{\xi >0:\rho _{f}^{-}(\xi )\leq t\}\}.
\end{multline}

Suppose first that $\gamma _{f}^{-}(t)=\sup \{\xi :\rho _{f}^{-}(\xi )\leq
t\}\leq 0.$ Accordingly, $\{\xi >0:\rho _{f}^{-}(\xi )\leq t\}=\emptyset $
and so $\sup \{\xi >0:\rho _{f}^{-}(\xi )\leq t\}=-\infty .$ By (\ref{28}), $%
\gamma _{f_{+}}^{-}(t)=\max \{0,-\infty \}=0=(\gamma _{f}^{-})_{+}(t).$
Suppose next that $\gamma _{f}^{-}(t)=\sup \{\xi :\rho _{f}^{-}(\xi )\leq
t\}>0,$ so that $\gamma _{f}^{-}(t)=\sup \{\xi >0:\rho _{f}^{-}(\xi )\leq
t\}.$ By (\ref{28}), $\gamma _{f_{+}}^{-}(t)=\max \{0,\gamma
_{f}^{-}(t)\}=(\gamma _{f}^{-})_{+}(t).$ This shows that $\gamma
_{f_{+}}^{-}=(\gamma _{f}^{-})_{+},$ whence $(f_{+}\check{)}=\check{f}_{+}$
by (\ref{26}).
\end{proof}

\begin{remark}
\label{rm4} By (\ref{25}) and Remark \ref{rm3}, $m_{\hat{f}}=m_{f}$ if and
only if there is a sequence of \emph{distinct} points $x_{n}$ such that $%
f(x_{n})\rightarrow m_{f}.$
\end{remark}

\section{Reverse inequalities for $f\barwedge g$\label{reverse1}}

We begin with the (trivial) converse of the Brunn-Minkowski inequality for
Euclidean balls.

\begin{lemma}
\label{lm14}If $B_{1}$ and $B_{2}$ are Euclidean balls in $\Bbb{R}^{N},$
then $\mu _{N}(B_{1}+B_{2})\leq 2^{N-1}(\mu _{N}(B_{1})+\mu _{N}(B_{2})).$
\end{lemma}

\begin{proof}
Call $r_{i}$ the radius of $B_{i},i=1,2.$ It is readily checked that $%
B_{1}+B_{2}$ is a ball with radius $r_{1}+r_{2}$ and the inequality simply
follows from $(r_{1}+r_{2})^{N}\leq 2^{N-1}(r_{1}^{N}+r_{2}^{N}).$
\end{proof}

In the next lemma, $M_{f}=M_{g}$ is \emph{not} needed (compare with Lemma 
\ref{lm4}).

\begin{lemma}
\label{lm15}If $f,g:\Bbb{R}^{N}\rightarrow [0,\infty ]$ are Borel
measurable, then 
\begin{equation*}
\int_{\Bbb{R}^{N}}f\barwedge g\leq 2^{N-1}\left( \int_{\Bbb{R}^{N}}\hat{f}%
+\int_{\Bbb{R}^{N}}\hat{g}\right) .
\end{equation*}
\end{lemma}

\begin{proof}
By (\ref{5}) and (\ref{6}) and since $f\barwedge g\geq 0$ is measurable
(Lemma \ref{lm2}), it follows that $\int_{\Bbb{R}^{N}}f\barwedge
g=\int_{0}^{\infty }\mu _{N}(F_{\xi }^{+}+G_{\xi }^{+})d\xi .$ Next, $F_{\xi
}^{+}+G_{\xi }^{+}\subset \overline{B}_{F_{\xi }^{+}}+\overline{B}_{G_{\xi
}^{+}}$ and so $\mu _{N}(F_{\xi }^{+}+G_{\xi }^{+})\leq \mu _{N}(\overline{B}
_{F_{\xi }^{+}}+\overline{B}_{G_{\xi }^{+}})\leq 2^{N-1}(\mu _{N}(\overline{B%
}_{F_{\xi }^{+}})+\mu _{N}(\overline{B}_{G_{\xi }^{+}})),$ where Lemma \ref
{lm14} was used. This yields $\int_{\Bbb{R}^{N}}f\barwedge g\leq
2^{N-1}\left( \int_{0}^{\infty }\mu _{N}(\overline{B}_{F_{\xi }^{+}})d\xi
+\int_{0}^{\infty }\mu _{N}(\overline{B}_{G_{\xi }^{+}})d\xi \right) .$ By
Theorem \ref{th11} (ii), the right-hand side is $2^{N-1}\left(
\int_{0}^{\infty }\mu _{N}(\hat{F}_{\xi }^{+}))d\xi +\int_{0}^{\infty }\mu
_{N}(\hat{G}_{\xi }^{+}))d\xi \right) ,$ which in turn equals $2^{N-1}\left(
\int_{\Bbb{R}^{N}}\hat{f}+\int_{\Bbb{R}^{N}}\hat{g}\right) $ because $\hat{f}
,\hat{g}\geq 0$ by Theorem \ref{th11} (iii).
\end{proof}

No variant of Lemma \ref{lm15} is true if $\hat{f}$ (or $\hat{g}$) is
replaced with $f$ (or $g$):

\begin{example}
\label{ex1}With $N=1,$ let $0<f\leq 1$ be integrable with $f(n)=1$ for every 
$n\in \Bbb{Z}$ and let $g=\chi _{(-1,1)}$ ($=\hat{g}$). Then, $f\barwedge
g=1,$ whence $\int_{\Bbb{R}}f\barwedge g=\infty $ but $f,g\in L^{1}.$
\end{example}

If $N=1,$ $f,g\geq 0$ are even and nonincreasing on $[0,\infty )$ and $%
M_{f}=M_{g}.$ Then, $\int_{\Bbb{R}}f\barwedge g=\int_{\Bbb{R}}f+\int_{\Bbb{R}%
}g$ by Theorem \ref{th11} (i) and Lemmas \ref{lm4} and \ref{lm15}. Is there
a different proof ? (If $f=g,$ this follows from $(f\barwedge f)(x)=f(x/2).$)

To go further, we need a simple property of Young functions.

\begin{lemma}
\label{lm16}If $\phi $ is a Young function and $f:\Bbb{R}^{N}\rightarrow [0,%
\infty ],$ then $(\phi (f)\hat{)}=\phi (\hat{f}).$
\end{lemma}

\begin{proof}
For brevity, we only give the proof in the more important case when $0<\phi
<\infty $ on $(0,\infty ),$ so that $\phi $ is continuous on $[0,\infty ]$
and has an inverse $\psi .$ The general case involves extra technicalities
that lengthen the exposition. Recall that $\rho _{f}^{+}(\xi )$ denotes the
radius of $\overline{B}_{F_{\xi }^{+}}.$

Since $\phi (f)\geq 0,$ it follows that $\rho _{\phi (f)}^{+}(\xi )=\infty $
if $\xi <0.$ If $\xi \geq 0,$ then $\{x\in \Bbb{R}^{N}:\phi (f(x))>\xi
\}=\{x\in \Bbb{R}^{N}:f(x)>\psi (\xi )\}=F_{\psi (\xi )}^{+},$ so that $\rho
_{\phi (f)}^{+}(\xi )=\rho _{f}^{+}(\psi (\xi )).$ Thus, by (\ref{21}) and (%
\ref{23}), $(\phi (f)\hat{)}(x)=\inf \{\xi \geq 0:\rho _{f}^{+}(\psi (\xi
))\leq |x|\}=\inf \{\phi (\eta ):\rho _{f}^{+}(\eta )\leq |x|\}=\phi (\inf
\{\eta \geq 0:\rho _{f}^{+}(\eta )\leq |x|\}).$ Now, $\inf \{\eta \geq
0:\rho _{f}^{+}(\eta )\leq |x|\}=\hat{f}(x)$ by (\ref{21}) and (\ref{23})
because $f\geq 0$ implies $\rho _{f}^{+}(\eta )=\infty $ if $\eta <0.$
\end{proof}

\begin{lemma}
\label{lm17}If $\phi $ is a Young function and $f:\Bbb{R}^{N}\rightarrow [0,%
\infty ]$ is measurable, then $||f||_{\phi }\leq ||\hat{f}||_{\phi }.$
\end{lemma}

\begin{proof}
Since $f\geq 0$ implies $\hat{f}\geq 0,$ it follows from Theorem \ref{th11}
(ii) and from $\mu _{N}(F_{\xi }^{+})\leq \mu _{N}(\overline{B}_{F_{\xi
}^{+}})$ that $\int_{\Bbb{R}^{N}}f=\int_{0}^{\infty }\mu _{N}(F_{\xi
}^{+})d\xi \leq \int_{0}^{\infty }\mu _{N}(\hat{F}_{\xi }^{+})d\xi =\int_{%
\Bbb{R}^{N}}\hat{f}.$ Upon replacing $f$ by $\phi (f)$ in this inequality,
it follows from Lemma \ref{lm16} that $\int_{\Bbb{R}^{N}}\phi (f)\leq \int_{%
\Bbb{R}^{N}}\phi (\hat{f}).$ Now, replace $f$ with $r^{-1}f$ where $r>0$ and
use Theorem \ref{th11} (v) to get $\int_{\Bbb{R}^{N}}\phi (r^{-1}f)\leq
\int_{\Bbb{R}^{N}}\phi (r^{-1}\hat{f})$ for every $r>0.$ By (\ref{13}), this
implies $||f||_{\phi }\leq ||\hat{f}||_{\phi }.$
\end{proof}

\begin{theorem}
\label{th18} If $f,g:\Bbb{R}^{N}\rightarrow [0,\infty ]$ are Borel
measurable and $\phi $ is a Young function, then $\phi (f),\phi (g)$ and $%
\phi (f\barwedge g)$ are measurable and nonnegative and\newline
\begin{equation}
||f\barwedge g||_{\phi }\leq 2^{N-1}(||\hat{f}||_{\phi }+||\hat{g}||_{\phi
}).  \label{29}
\end{equation}
\newline
\end{theorem}

\begin{proof}
For the measurability of $\phi (f),\phi (g)$ and $\phi (f\barwedge g),$ see
Lemma \ref{lm6} and Remark \ref{rm1}. That all three are nonnegative is
trivial. In the proof of Lemma \ref{lm6}, we already established that $\phi
(f\barwedge g)=\phi (f)\barwedge \phi (g)$. Therefore, by Lemma \ref{lm15}
for $\phi (f)$ and $\phi (g)$ and by Lemma \ref{lm16}, we get 
\begin{equation}
\int_{\Bbb{R}^{N}}\phi (f\barwedge g)\leq 2^{N-1}\left( \int_{\Bbb{R}%
^{N}}\phi (\hat{f})+\int_{\Bbb{R}^{N}}\phi (\hat{g})\right) .  \label{30}
\end{equation}

Suppose first that\textit{\ }$||\hat{f}||_{\phi }>0$ and $||\hat{g}||_{\phi
}>0.$ Since (\ref{29}) is trivial otherwise, we may and shall assume $||\hat{%
f}||_{\phi }<\infty $ and $||\hat{g}||_{\phi }<\infty .$ If so, $0<r:=||\hat{%
f}||_{\phi }+||\hat{g}||_{\phi }<\infty $ and the inequality (\ref{30}) for $%
r^{-1}f$ and $r^{-1}g$ is (use Theorem \ref{th11} (v)) 
\begin{equation*}
2^{1-N}\int_{\Bbb{R}^{N}}\phi (r^{-1}(f\barwedge g))\leq \int_{\Bbb{R}%
^{N}}\phi (r^{-1}\hat{f})+\int_{\Bbb{R}^{N}}\phi (r^{-1}\hat{g}).
\end{equation*}
With $\lambda :=r^{-1}||\hat{f}||_{\phi }\in (0,1),$ so that $r^{-1}\hat{f}%
=\lambda ||\hat{f}||_{\phi }^{-1}\hat{f}$ and that $r^{-1}\hat{g}=(1-\lambda
)||\hat{g}||_{\phi }^{-1}\hat{g},$ this reads 
\begin{equation}
2^{1-N}\int_{\Bbb{R}^{N}}\phi (r^{-1}(f\barwedge g))\leq \int_{\Bbb{R}%
^{N}}\phi \left( \lambda ||\hat{f}||_{\phi }^{-1}\hat{f}\right) +\int_{\Bbb{R%
}^{N}}\phi \left( (1-\lambda )||\hat{g}||_{\phi }^{-1}\hat{g}\right) .
\label{31}
\end{equation}
Since $\phi $ is convex and $\phi (0)=0,$ then $\phi (\mu \tau )\leq \mu
\phi (\tau )$ when $\tau \geq 0$ and $0\leq \mu \leq 1.$ The choices $\mu
=2^{1-N}$ in the left-hand side of (\ref{31}) and, next, $\mu =\lambda $ and 
$\mu =1-\lambda $ in its right-hand side, yield 
\begin{equation*}
\int_{\Bbb{R}^{N}}\phi \left( 2^{1-N}r^{-1}(f\barwedge g)\right) \leq
\lambda \int_{\Bbb{R}^{N}}\phi \left( ||\hat{f}||_{\phi }^{-1}\hat{f}\right)
+(1-\lambda )\int_{\Bbb{R}^{N}}\phi \left( ||\hat{g}||_{\phi }^{-1}\hat{g}%
\right) .
\end{equation*}
By (\ref{15}), it follows that $\int_{\Bbb{R}^{N}}\phi \left(
2^{1-N}r^{-1}(f\barwedge g)\right) \leq 1$ and so, by (\ref{13}), $%
||f\barwedge g||_{\phi }\leq 2^{N-1}r=2^{N-1}(||\hat{f}||_{\phi }+||\hat{g}
||_{\phi }),$ as claimed in (\ref{29}).

To complete the proof, suppose now that $||\hat{f}||_{\phi }=0$ or $||\hat{g}%
||_{\phi }=0.$ By symmetry, we may and shall assume that $||\hat{f}||_{\phi
}=0,$ whence $\hat{f}=0$ a.e. Since $\hat{f}(x)$ is a nonincreasing \emph{and%
} right-continuous function of $|x|,$ it follows that $\hat{f}=0.$ Thus, by
Lemma \ref{lm12}, either $f=0$ or there are $x_{0}\in \Bbb{R}^{N}$ and $%
0<z\leq \infty $ such that $f(x_{0})=z$ and $f(x)=0$ when $x\neq x_{0}.$

If $f=0,$ then $f\barwedge g=0$ since $g\geq 0$ and (\ref{29}) is trivial.
If $f(x_{0})=z>0$ and $f(x)=0$ when $x\neq x_{0},$ a straightforward
calculation shows that $(f\barwedge g)(x)=\min \{g(x-x_{0}),z\}.$ In
particular, $0\leq (f\barwedge g)\leq g(\cdot -x_{0}),$ whence $||f\barwedge
g||_{\phi }\leq ||g(\cdot -x_{0})||_{\phi }=||g||_{\phi }\leq ||\hat{g}
||_{\phi },$ the latter by Lemma \ref{lm17}. This proves (\ref{29}).
\end{proof}

Since $(f\barwedge f)(x)\geq f(x/2),$ equality holds in (\ref{29}) if $%
L_{\phi }=L^{1}$ and $0\leq g=f=\hat{f}\in L^{1}.$ Thus, $2^{N-1}$ is best
possible among all the constants independent of $\phi .$ On the other hand,
even when $L_{\phi }=L^{\infty },$ (\ref{29}) is trivial only under
additional assumptions, namely, $M_{f}=\limfunc{ess}\sup f,M_{g}=\limfunc{ess%
}\sup g$ and $\min \{M_{f},M_{g}\}=\limfunc{ess}\sup (f\barwedge g)$ (which
is \emph{not} implied by the former two, see Example \ref{ex2} below). If
so, 
\begin{equation}
||f\barwedge g||_{\infty }=\min \{||f||_{\infty },||g||_{\infty }\},
\label{32}
\end{equation}
which is much better than (\ref{29}) since $||f||_{\infty }\leq ||\hat{f}%
||_{\infty }$ and $||g||_{\infty }\leq ||\hat{g}||_{\infty }$ by Lemma \ref
{lm17}. However, (\ref{32}) is false without the extra assumptions mentioned
above and then (\ref{29}) is no longer trivial in $L^{\infty }.$ How badly (%
\ref{32}) may fail is shown in:

\begin{example}
\label{ex2}Let $\mathcal{C}\subset [0,1]$ be the Cantor set. It is not hard
to see that $(x-\mathcal{C})\cap \mathcal{C}\neq \emptyset $ for every $x\in
[0,2]$ (notice that $(x-\mathcal{C}_{k})\cap \mathcal{C}_{k}\neq \emptyset $
for every $k\in \Bbb{N},$ where $\mathcal{C}_{1}=[0,1]$ and $\mathcal{C}%
_{k+1}\subset \mathcal{C}_{k}$ is obtained by removing the open middle
thirds of the intervals in $\mathcal{C}_{k}$). As a result, if $f=\infty $
on $\mathcal{C}$ and $0$ outside, then $f$ is Borel measurable, the
right-hand side of (\ref{32}) with $g=f$ is $0,$ but since $f\barwedge
f=\infty $ on $[0,2],$ its left-hand side is $\infty .$
\end{example}

\section{Reverse inequalities for $f\Box g$ and $f\veebar g$\label{reverse2}}

With the help of Theorem \ref{th18}, it is a simple matter to prove a
converse of Theorem \ref{th7}.

\begin{theorem}
\label{th19} Suppose that $f,g:\Bbb{R}^{N}\rightarrow (-\infty ,\infty ]$
are Borel measurable, that $m_{f},m_{g}\in \Bbb{R}$ and that $%
m_{f}+m_{g}\geq 0.$ Then $f\Box g\geq 0,\check{f}-m_{f,g}\geq 0,\check{g}%
+m_{f,g}\geq 0$ and 
\begin{equation}
||(f\Box g)^{-1}||_{\phi }\leq 2^{N-1}(||(\check{f}-m_{f,g})^{-1}||_{\phi
}+||(\check{g}+m_{f,g})^{-1}||_{\phi }),  \label{33}
\end{equation}
for every Young function $\phi .$
\end{theorem}

\begin{proof}
Use $f-m_{f,g}\geq 0,g+m_{f,g}\geq 0$ along with $f\Box g=(f-m_{f,g})\Box
(g+m_{f,g})\geq (f-m_{f,g})\veebar (g+m_{f,g})\geq 0$ to get $\check{f}
-m_{f,g}\geq 0$ and $\check{g}+m_{f,g}\geq 0$ (by Theorem \ref{th13} (iii)
and (iv)) as well as $0\leq (f\Box g)^{-1}\leq ((f-m_{f,g})\veebar
(g+m_{f,g}))^{-1}=(f-m_{f,g})^{-1}\barwedge (g+m_{f,g})^{-1}.$ This yields 
\begin{eqnarray*}
||(f\Box g)^{-1}||_{\phi } &\leq &||(f-m_{f,g})^{-1}\barwedge
(g+m_{f,g})^{-1}||_{\phi }\leq \\
&&2^{N-1}(||((f-m_{f,g})^{-1}\hat{)}||_{\phi }+||((g+m_{f,g})^{-1}\hat{)}
||_{\phi },
\end{eqnarray*}
where Theorem \ref{th18} was used for the second inequality.

Next, $((f-m_{f,g})^{-1}\hat{)}=(\check{f}-m_{f,g})^{-1}$ and $%
((g+m_{f,g})^{-1}\hat{)}=(\check{g}+m_{f,g})^{-1}$ by Theorem \ref{th13}
(viii) and (iv) (in that order), which proves (\ref{33}).
\end{proof}

The comments after Theorem \ref{th18} may be repeated: Without extra
assumptions, (\ref{33}) is not trivial even when $L_{\phi }=L^{\infty }.$
The proof of Theorem \ref{th19} shows that, more generally, if $z\in \Bbb{R}$
and $-m_{g}\leq z\leq m_{f}$ (whence $m_{f}+m_{g}\geq 0$), 
\begin{equation}
||(f\Box g)^{-1}||_{\phi }\leq 2^{N-1}(||(\check{f}-z)^{-1}||_{\phi }+||(%
\check{g}+z)^{-1}||_{\phi }).  \label{34}
\end{equation}
However, (\ref{33}) is optimal (to prove $(f\Box g)^{-1}\in L_{\phi }$)
among all the inequalities (\ref{34}). Indeed, if the right-hand side of (%
\ref{34}) is finite for some $z$ as above, the right-hand side of (\ref{33})
is also finite\footnote{%
While mostly true, the converse may fail when $m_{f}+m_{g}>0$ and $%
z=m_{f}=m_{\check{f}}$ or $z=-m_{g}=-m_{\check{g}}.$}. To see this, assume
first $m_{f}+m_{g}=0.$ Then, $z=m_{f}=-m_{g}=m_{f,g}$ is the only possible
choice in (\ref{34}) and (\ref{33}) is optimal by default. Suppose now that $%
m_{f}+m_{g}>0$ and note that, by Theorem \ref{th13} (iii), $\check{f}%
-m_{f,g} $ and $\check{g}+m_{f,g}$ are both bounded below by $%
(m_{f}+m_{g})/2>0.$ Thus, $(\check{f}-m_{f,g})^{-1}$ $\in L^{\infty }$ and
so, if $\check{f}-z\geq 0$ and $(\check{f}-z)^{-1}\in L_{\phi }$ for some $%
z, $ then $(\check{f}-m_{f,g})^{-1}=[1-(z-m_{f,g})(\check{f}-m_{f,g})^{-1}](%
\check{f}-z)^{-1}\in L_{\phi }$ since $1-(z-m_{f,g})(\check{f}%
-m_{f,g})^{-1}\in L^{\infty }.$ Likewise, if $\check{g}+z\geq 0$ and $(%
\check{g}+z)^{-1}\in L_{\phi },$ then $(\check{g}+m_{f,g})^{-1}\in L_{\phi
}. $

It is not clear whether $2^{N-1}$ is best possible in (\ref{33}), among all
the constants independent of $\phi .$ The remark that $(f\Box f)(x)\leq
2f(x/2)$ and the choice $g=f=\check{f}\geq 0$ with $f^{-1}\in L^{1}=L_{\phi
} $ only shows that the best constant is at least $2^{N-2}.$

There is also a converse of Corollary \ref{cor8}:

\begin{corollary}
\label{cor20}Suppose that $f,g:\Bbb{R}^{N}\rightarrow [-\infty ,\infty ]$
are Borel measurable, that $g\geq 0$ and that $f\not{\equiv}\infty ,g\not%
{\equiv}\infty .$ Then, $f\veebar g\geq 0$ and \newline
\begin{equation}
||(f\veebar g)^{-1}||_{\phi }\leq 2^{N}\left( ||(\check{f}%
_{+}-m_{f_{+},g})^{-1}||_{\phi }+||(\check{g}+m_{f_{+},g})^{-1}||_{\phi
}\right) .  \label{35}
\end{equation}
for every Young function $\phi .$
\end{corollary}

\begin{proof}
Since $g$ is nonnegative, $f\veebar g=f_{+}\veebar g\geq (f_{+}\Box
g)/2=(f_{+}/2)\Box (g/2)\geq 0.$ Also, $0\leq m_{f_{+}},m_{g}<\infty $ since 
$f\not{\equiv}\infty $ and $g\not{\equiv}\infty .$ Thus, it suffices to use
Theorem \ref{th19} with $f_{+}/2$ and $g/2$ along with Theorem \ref{th13}
(v) and (ix).
\end{proof}

In many cases, Theorem \ref{th19} and Corollary \ref{cor20} can be used to
prove that $(f\Box g)^{-1}\in L_{\phi }$ or $(f\veebar g)^{-1}\in L_{\phi }$
without any calculation of $\check{f}$ or $\check{g},$ because there are
easily verifiable sufficient conditions for the finiteness of the right-hand
sides of (\ref{33}) and (\ref{35}). The simplest one is given in the
following lemma.

\begin{lemma}
\label{lm21}Given $f:\Bbb{R}^{N}\rightarrow (-\infty ,\infty ]$ with $%
m_{f}>-\infty $ (i.e., $f$ bounded below), suppose that there are constants $%
c>0$ and $\alpha >0$ such that $f(x)\geq h(x):=c|x|^{\alpha }$ for $|x|$
large enough. Then, for every $z<m_{f},$ 
\begin{equation}
||(\check{f}-z)^{-1}||_{\phi }\leq 2||h^{-1}||_{\phi ,\Bbb{R}^{N}\backslash
B}+(m_{f}-z)^{-1}||\chi _{B}||_{\phi }.  \label{36}
\end{equation}
for every open ball $B$ centered at the origin such that\footnote{%
Since $\lim_{|x|\rightarrow \infty }h(x)=\infty ,$ the existence of $B$ is
not an issue.} $f\geq h\geq 2z$ outside $B$ and every Young function $\phi .$
If $\phi $ is invertible with inverse $\psi ,$ this also reads 
\begin{equation}
||(\check{f}-z)^{-1}||_{\phi }\leq 2||h^{-1}||_{\phi ,\Bbb{R}^{N}\backslash
B}+(m_{f}-z)^{-1}[\psi (\mu _{N}(B)^{-1})]^{-1}.  \label{37}
\end{equation}
In particular, if $h^{-1}\in L_{\phi }(\Bbb{R}^{N}\backslash B),$ then $(%
\check{f}-z)^{-1}\in L_{\phi }.$ (Case in point: Since $h$ is explicitly
known, $h^{-1}\in L_{\phi }(\Bbb{R}^{N}\backslash B)$ can often be checked
by a calculation.)
\end{lemma}

\begin{proof}
Let $B$ be as in the theorem. By the ``furthermore'' part of Theorem \ref
{th13} (vii), $\check{f}\geq h\geq 2z$ outside $B.$ Thus, $h\leq 2(\check{f}%
-z),$ whence $(\check{f}-z)^{-1}\leq 2h^{-1},$ outside $B.$ Meanwhile, by
Theorem \ref{th13} (iii), $(\check{f}-z)^{-1}\leq (m_{\check{f}%
}-z)^{-1}<(m_{f}-z)^{-1}<\infty $ in $B$ (and everywhere else) and (\ref{36}%
) follows. To get (\ref{37}) when $\phi $ is invertible with inverse $\psi ,$
just notice that $||\chi _{B}||_{\phi }=[\psi (\mu _{N}(B)^{-1})]^{-1}$ by (%
\ref{13}).
\end{proof}

From the proof of Lemma \ref{lm21}, $|x|^{\alpha }$ can be replaced with a
function $h(x)$ satisfying general conditions. For example, if $L_{\phi
}=L^{p}$ with $p\geq 1,$ Lemma \ref{lm21} yields $(\check{f}-z)^{-1}\in
L^{p} $ if $\alpha >N/p$ but, if $f(x)\geq c|x|^{\alpha }(\log |x|)^{\beta }$
for large $|x|,$ the choice of any continuous strictly increasing function $h
$ of $|x|$ that coincides with $c|x|^{\alpha }(\log |x|)^{\beta }$ for $%
|x|\geq 1$ (say) shows that $(\check{f}-z)^{-1}\in L^{p}$ in the limiting
case $\alpha =N/p$ if $\beta >p^{-1}.$

\begin{remark}
The proof of Lemma \ref{lm21} also shows that, more generally, (\ref{36}) is
true with $m_{f}$ replaced with $m_{\check{f}}$ and that this requires only $%
m_{\check{f}}>-\infty .$ This may occasionally be useful, but rarely (see
Remark \ref{rm4}).
\end{remark}

It is more delicate to extend Lemma \ref{lm21} when $z=m_{f}=m_{\check{f}}.$
The extra difficulty is that $(\check{f}-m_{f})^{-1}\notin L_{loc}^{\infty }$
(because $m_{\check{f}}=\limfunc{ess}\inf \check{f};$ see Theorem \ref{th13}
(iii)), so that local integrability becomes an issue. This requires further
investigation. We only mention without proof (and will not use later) that
if $m_{f}$ is a \emph{unique} \emph{\ }and \emph{nondegenerate} minimum of $%
f $ (plus a mild technical condition) then $(\check{f}-m_{f})^{-1}\in
L_{loc}^{p}$ if $N\geq 3$ and $1\leq p<N/2$.

A direct application of Lemma \ref{lm21} yields the following sample result.

\begin{theorem}
\label{th22}Given $f,g:\Bbb{R}^{N}\rightarrow (-\infty ,\infty ],$ suppose
that there are constants $c>0$ and $\alpha >0$ such that $f(x)\geq
c|x|^{\alpha }$ and $g(x)\geq c|x|^{\alpha }$ for $|x|$ large enough. Then:%
\newline
(i) If $m_{f},m_{g}\in \Bbb{R}$ and $m_{f}+m_{g}>0,$ then $(f\Box g)^{-1}\in
L^{p}$ if $p\geq 1$ and $p>N/\alpha .$ \newline
(ii) If $g\geq 0$ and $m_{f_{+}}+m_{g}>0,$ then $(f\veebar g)^{-1}\in L^{p}$
if $p\geq 1$ and $p>N/\alpha .$
\end{theorem}

\begin{proof}
(i) Just notice that $m_{f,g}<m_{f}$ ($\leq m_{\check{f}}$) and $%
-m_{f,g}<m_{g}$ ($\leq m_{\check{g}}$) since $m_{f}+m_{g}>0$ and use Lemma 
\ref{lm21} with $L_{\phi }=L^{p}$ and Theorem \ref{th19}.

(ii) If $m_{f_{+}}=\infty $ or $m_{g}=\infty ,$ then $f\veebar g=\infty $
and the result is trivial. From now on, $m_{f_{+}},m_{g}<\infty ,$ whence $%
m_{f_{+}},m_{g}\in \Bbb{R}$ (because $f_{+},g\geq 0$). By Theorem \ref{th13}
(ix), $\check{f}_{+}=(f_{+}\check{)},$ whereas $m_{f_{+},g}<m_{f_{+}}$ and $%
-m_{f_{+},g}<m_{g}$ since $m_{f_{+}}+m_{g}>0.$ Now, use Lemma \ref{lm21} and
Corollary \ref{cor20} with $L_{\phi }=L^{p}.$
\end{proof}

Lemma \ref{lm21} and its aforementioned variants yield generalizations of
Theorem \ref{th22} to all Orlicz spaces. The proof of Theorem \ref{th22}
does not use the estimate (\ref{36}) but the next theorem, relevant to the
results in the next section, does. As explained after the proof, there is a
good reason to confine attention to $L^{p}$ spaces.

\begin{theorem}
\label{th23}Suppose that $f,g:\Bbb{R}^{N}\rightarrow (-\infty ,\infty ]$ are
Borel measurable, that $m_{f}\geq 0,m_{g}>-\infty $ and that $m_{g}>0$ if $%
m_{f}=0.$ Given $1\leq p\leq \infty ,$ suppose also that for some $\alpha
>N/p,$ there are constants $c>0$ and $\alpha >0$ such that $f(x)\geq
c|x|^{\alpha }$ for $|x|$ large enough. For $t>0,$ set 
\begin{equation*}
f_{t}(x):=tf(t^{-1}x),
\end{equation*}
so that $m_{f_{t}}=tm_{f}.$ Lastly, assume $(\check{g}-z)^{-1}\in L^{p}$
when $z<m_{g}$ (e.g., if $g(x)\geq c|x|^{\alpha }$ for $|x|$ large enough by
Lemma \ref{lm21}, but this is not necessary). \newline
Then, $(f_{t}\Box g)^{-1}\in L^{p}$ when $tm_{f}+m_{g}>0$ (i.e., $t>0$ if $%
m_{f}=0$ and $t>-m_{f}^{-1}m_{g}$ if $m_{f}>0$) and:\newline
(i) If $m_{f}=0,$ then $||(f_{t}\Box g)^{-1}||_{p}=O(t^{N/p})$ as $%
t\rightarrow \infty $.\newline
(ii) If $m_{f}>0$ and $p>N,$ then $\lim_{t\rightarrow \infty }||(f_{t}\Box
g)^{-1}||_{p}=0.$\newline
(iii) If $m_{f}>0$ and $p\geq N,$ then $||(f_{t}\Box
g)^{-1}||_{p}=O(t^{-1+N/p})$ as $t\rightarrow \infty .$
\end{theorem}

\begin{proof}
As in Lemma \ref{lm21}, set $h(x):=c|x|^{\alpha }$ and, in (\ref{36}), let $%
B $ be a ball centered at the origin such that $\check{f}\geq h$ and $h\geq
m_{f}+|m_{g}|\geq 2m_{f,g}=m_{f}-m_{g}$ outside $B.$ Evidently, $(\check{f}
)_{t}\geq h_{t}$ ($:=th(t^{-1}\cdot )$) outside $tB$ and, by Theorem \ref
{th13} (iv) and (v), $(f_{t}\check{)}=(\check{f})_{t}.$ Thus, $(\check{f}
)_{t}\geq h_{t}$ outside $tB.$ Furthermore, if $t\geq 1,$ then $h_{t}\geq
tm_{f}+|m_{g}|\geq 2m_{f_{t},g}=tm_{f}-m_{g}$ outside $tB.$

Since $h_{t}(x)=t^{1-\alpha }c|x|^{\alpha },$ it follows from the above that
the estimate (\ref{36}) can be used with $f,h$ and $B$ replaced with $%
f_{t},h_{t}$ and $tB,$ respectively, and with $%
z=m_{f_{t},g}=(tm_{f}-m_{g})/2,$ provided that $t\geq 1$ and that $%
tm_{f}+m_{g}>0$ (so that $m_{f_{t},g}<m_{f_{t}}=tm_{f}$), which holds for
large $t.$ Accordingly, 
\begin{multline}
||(\check{f}_{t}-m_{f_{t},g})^{-1}||_{p}\leq  \label{38} \\
2t^{-1+N/p}||h^{-1}||_{p,\Bbb{R}^{N}\backslash
B}+2(tm_{f}+m_{g})^{-1}t^{N/p}\mu _{N}(B)^{1/p},
\end{multline}
where $||\chi _{tB}||_{p}=t^{N/p}\mu _{N}(B)^{1/p}$ was used. Since $%
||h^{-1}||_{p,\Bbb{R}^{N}\backslash B}<\infty $ by the choice $\alpha >N/p,$
it follows that $(\check{f}_{t}-m_{f_{t},g})^{-1}\in L^{p}.$ Also, $(\check{g%
}+m_{f_{t},g})^{-1}\in L^{p}$ since $-m_{f_{t},g}=(m_{g}-tm_{f})/2<m_{g}$
when $tm_{f}+m_{g}>0$ and since it is assumed that $(\check{g}-z)^{-1}\in
L^{p}$ when $z<m_{g}.$ Thus, by (\ref{33}), $(f_{t}\Box g)^{-1}\in L^{p}.$

The estimates (i), (ii) and (iii) follow from (\ref{33}) and (\ref{38}) and
from the remarks that (a) if $m_{f}=0,$ then ($m_{g}>0$ and) $(\check{g}
+m_{f_{t},g})^{-1}=(\check{g}-m_{g}/2)^{-1}$ is independent of $t$ and (b)
if $m_{f}>0,$ then $\lim_{t\rightarrow \infty }||(\check{g}
+m_{f_{t},g})^{-1}||_{p}=0$ by dominated convergence if $p<\infty $ and by $%
\check{g}+m_{f_{t},g}\geq (tm_{f}+m_{g})/2$ if $p=\infty .$
\end{proof}

Similar estimates hold when $f$ is replaced with $f_{t}$ in Corollary \ref
{cor20} and estimates can also be worked out in other spaces $L_{\phi },$
but the technicalities depend on $\phi .$ For instance, while Theorem \ref
{th23} remains true if $L^{p}$ is replaced with $L^{1}+L^{p},$ the proof is
substantially more demanding (recall that $L^{1}+L^{p}=L_{\phi }$ with $\phi
(\tau )=\tau ^{p}$ in $[0,1]$ and $\phi (\tau )=p\tau +1-p$ in $(1,\infty )$
if $1\leq p<\infty $ and $L^{1}+L^{\infty }=L_{\phi }$ with $\phi (\tau )=0$
in $[0,1]$ and $\phi (\tau )=t-1$ in $(1,\infty );$ see e.g. \cite{Hu81})).
Choices of $h$ other than $h(x)=c|x|^{\alpha }$ often lead to challenging
calculations.

\section{Application to the Hamilton-Jacobi equations\label{application}}

We shall now apply the results of the previous sections to the
Hamilton-Jacobi equations in their simplest form (see Subsection \ref{other}
for a variant) 
\begin{equation}
\left\{ 
\begin{array}{l}
u_{t}+H(\nabla u)=0\text{ on }(0,\infty )\times \Bbb{R}^{N}, \\ 
u(0,\cdot )=g\text{ on }\Bbb{R}^{N},
\end{array}
\right.  \label{39}
\end{equation}
where the Hamiltonian $H$ and the initial value $g$ are given functions on $%
\Bbb{R}^{N}.$

Roughly speaking, when the Hamiltonian $H$ (initial condition $g$) is
convex, the Hopf-Lax formula (Hopf formula) provides a solution of (\ref{39}
). In both cases, various additional conditions are required of $H$ and $g$
and, as always, what constitutes a solution is somewhat flexible. While the
more recent work focuses on viscosity solutions, other definitions exist as
well.

Throughout this section, we assume that $g,H:\Bbb{R}^{N}\rightarrow (-\infty
,\infty ],$ that $m_{g}\in \Bbb{R}$ (hence $g\not{\equiv}\infty $), $H\not%
{\equiv}\infty $ and that $g$ is Borel measurable. Further assumptions will
be introduced when needed. It is once and for all understood that $t>0.$

We denote by $(tH)^{*}$ the Legendre-Fenchel conjugate of $tH,$ that is, 
\begin{equation*}
(tH)^{*}(x):=\sup_{y\in \Bbb{R}^{N}}(x\cdot y-tH(y))=tH^{*}(t^{-1}x)
\end{equation*}
Since $(tH)^{*}$ is always lsc, it is Borel measurable.

\subsection{Solutions by the Hopf-Lax formula\label{lax}}

In this subsection, $H$ is convex and $H^{**}(0)\in \Bbb{R}.$ The Hopf-Lax
formula (Hopf \cite{Ho65}, Lax \cite{La57}) 
\begin{equation}
u(t,\cdot )=(tH)^{*}\Box g,  \label{40}
\end{equation}
is known to give a solution of (\ref{39}) under various conditions about $g.$
That $g$ is real-valued and continuous is a common assumption; see Bardi and
Faggian \cite{BaFa98} and the references therein. The case when $g$ is lsc
and not everywhere finite was considered by Imbert \cite{Im01} and
Str\"{o}mberg \cite{St02}. Chen and Su \cite{ChSu03} show that (\ref{40}) is
a solution when $g$ is real-valued, a.e. continuous and satisfies a
condition weaker than upper semicontinuity. Undoubtedly, other options can
be found in the literature.

The inequality (\ref{19}) in Theorem \ref{th7} can be used with $f=(tH)^{*}$
if $m_{(tH)^{*}}+m_{g}\geq 0.$ That $m_{g}>-\infty $ was assumed earlier,
whereas $m_{(tH)^{*}}=$ $-tH^{**}(0)\in \Bbb{R}.$ As a result, the condition 
$m_{(tH)^{*}}+m_{g}\geq 0$ is simply 
\begin{equation}
-tH^{**}(0)+m_{g}\geq 0,  \label{41}
\end{equation}
so that $m_{(tH)^{*},g}$ (see (\ref{18})) is given by 
\begin{equation*}
m_{(tH)^{*},g}=-(tH^{**}(0)+m_{g})/2.
\end{equation*}

Thus, assuming (\ref{41}), the corresponding inequality (\ref{19}) 
\begin{multline}
||(2(tH)^{*}+tH^{**}(0)+m_{g})^{-1}||_{\phi
}+||(2g-tH^{**}(0)-m_{g})^{-1}||_{\phi }\leq  \label{42} \\
2||u(t,\cdot )^{-1}||_{\phi }
\end{multline}
and the reverse inequality (\ref{33}) of Theorem \ref{th19} 
\begin{multline}
||u(t,\cdot )^{-1}||_{\phi }\leq  \label{43} \\
2^{N-1}\left( ||(2((tH)^{*}\check{)}+tH^{**}(0)+m_{g})^{-1}||_{\phi }+||(2%
\check{g}-tH^{**}(0)-m_{g})^{-1}||_{\phi }\right) ,
\end{multline}
hold for every Young function $\phi .$

Given $\alpha >1,$ call $\alpha ^{\prime }:=\alpha /(\alpha -1)>1$ the
H\"{o}lder conjugate of $\alpha .$ It is easily checked and certainly
folklore that if there is a constant $d>0$ such that $H(x)\leq d|x|^{\alpha
^{\prime }}$ for $|x|$ large enough, then $H^{*}(x)\geq c|x|^{\alpha }$ for $%
|x|$ large enough. Consistent with Theorem \ref{th22}, it follows that if
also $g(x)\geq c|x|^{\alpha }$ for $|x|$ large enough, then $u(t,\cdot
)^{-1}\in L^{p}$ for every $p\geq 1,p>N/\alpha .$

Furthermore, since $(tH)^{*}=tH^{*}(t^{-1}\cdot ),$ Theorem \ref{th23} gives
estimates for $||u(t,\cdot )^{-1}||_{p}$ as $t\rightarrow \infty $ if $%
H^{**}(0)<0$ or if $H^{**}(0)=0$ and $m_{g}>0.$ (If $H^{**}(0)>0,$ then (\ref
{43}) breaks down when $t>m_{g}/H^{**}(0).$) The accuracy (or possible lack
thereof) of these estimates can be evaluated by using the inequality (\ref
{42}) with $||\cdot ||_{\phi }=||\cdot ||_{p}.$

We now look at two classical examples in more detail. In both cases, $%
H^{**}(0)=H(0)=0$ will make the inequalities simpler, but confines the
discussion to $m_{g}\geq 0.$

\begin{example}
\label{ex3}Suppose that $H(x):=|x|^{2}/2,$ so that $(tH)^{*}(x)=|x|^{2}/2t$
and (\ref{41}) boils down to $m_{g}\geq 0$ since $H^{**}=H.$ If $1\leq
p<\infty ,$ a quick calculation shows that $%
||(2(tH)^{*}+m_{g})^{-1}||_{p}=A(p,m_{g})t^{N/2p}$ where $A(p,m_{g})>0$ is a
constant which is finite if and only if $m_{g}>0$ and $p>N/2.$ Accordingly,
from (\ref{42}), $u(t,\cdot )^{-1}\notin L^{p}$ for any $t>0$ if $m_{g}=0$
or if $N\geq 2$ and $1\leq p\leq N/2,$ even if $u(0,\cdot )^{-1}=g^{-1}\in
L^{p}.$\newline
\quad Assume now $m_{g}>0$ and $p>N/2.$ By (\ref{42}) with $H^{**}(0)=0$ and
since $0\leq g\leq 2g-m_{g}\leq 2g,$ 
\begin{equation}
||u(t,\cdot )^{-1}||_{p}\geq c(t^{N/2p}+||g^{-1}||_{p}),  \label{44}
\end{equation}
for some constant $c>0$ depending only upon $N,p$ and $m_{g}.$ Thus, once
again, $u(t,\cdot )^{-1}\notin L^{p}$if $g^{-1}\notin L^{p}.$\newline
Conversely, since $((tH)^{*}\check{)}=(tH)^{*}$ by Theorem \ref{th13} (i),
it follows from (\ref{43}) and from $0\leq \check{g}\leq 2\check{g}%
-m_{g}\leq 2\check{g}$ (since $m_{\check{g}}\geq m_{g};$ see Theorem \ref
{th13} (iii)) that $u(t,\cdot )^{-1}\in L^{p}$ for every $t>0$ if $\check{g}
^{-1}\in L^{p}$ and that $||u(t,\cdot )^{-1}||_{p}\leq C(t^{N/2p}+||\check{g}%
^{-1}||_{p}),$ where $C>0$ depends only upon $N,p$ and $m_{g}.$ In
particular, $||u(t,\cdot )^{-1}||_{p}=O(t^{N/2p})$ as $t\rightarrow \infty $
(which is sharp because of (\ref{44})). Note that the general estimate of
Theorem \ref{th23} gives only the less precise $O(t^{N/p}).$\newline
\quad Even though $u(t,\cdot )^{-1}\notin L^{p}$ if $m_{g}=0,$ this does not
preclude $u(t,\cdot )^{-1}\in L_{\phi }$ for Orlicz spaces outside the $%
L^{p} $ scale. For instance, if $L_{\phi }=L^{1}+L^{p},$ a simple calculation%
\footnote{%
The formula for $\phi $ was given at the end of the previous section.} shows
that $(2(tH)^{*}+m_{g})^{-1}\in L^{1}+L^{p}$ if $p>N/2,$ with the
restriction $N\geq 3$ if $m_{g}=0.$ If so, by (\ref{43}), $u(t,\cdot
)^{-1}\in L^{1}+L^{p}$ if $\check{g}^{-1}\in L^{1}+L^{p}$ and, by another
calculation, $||u(t,\cdot )^{-1}||_{L^{1}+L^{p}}=O(t)$ as $t\rightarrow
\infty $ (optimal by (\ref{42})). In particular, this holds if $\check{g}%
^{-1}\in L^{q}$ with $1\leq q\leq p$ and $p>N/2.$ This complements the $%
L^{p} $ discussion above, even when $m_{g}>0.$
\end{example}

\begin{example}
\label{ex4}Suppose that $H(x)=|x|,$ so that $(tH)^{*}$ is the indicator
function of the closed ball $\overline{B}(0,t).$ Thus, the formula (\ref{40}
) is simply $u(t,x)=\inf_{|y|\leq t}g(x-y).$\newline
Once again, $H^{**}=H,$ so that (\ref{41}) amounts to $m_{g}\geq 0$ and, if $%
\phi $ is a Young function, then (with $\omega _{N}:=\mu _{N}(B(0,1)$) 
\begin{equation}
||(2(tH)^{*}+m_{g})^{-1}||_{\phi }=\inf \left\{ r>0:\phi \left(
r^{-1}m_{g}^{-1}\right) \omega _{N}t^{N}\leq 1\right\} ,  \label{45}
\end{equation}
which is $\infty $ if $m_{g}=0.$ Thus, by (\ref{42}), $u(t,\cdot %
)^{-1}\notin L_{\phi }$ for any $t>0$ and any $\phi $ if $m_{g}=0,$ even if $%
u(0,\cdot )^{-1}=g^{-1}\in L_{\phi }.$ (Since $u(t,x)^{-1}=\sup_{|y|\leq %
t}g(x-y)^{-1},$ this also follows from the remark that for every $%
\varepsilon >0,u(t,\cdot )^{-1}\geq \varepsilon ^{-1}$ on some ball of
radius $t.$) Assuming from now on that $m_{g}>0,$ it follows from (\ref{45})
that if $\phi $ has an inverse $\psi $ on $[0,\infty ],$ then $%
||(2(tH)^{*}+m_{g})^{-1}||_{\phi }=m_{g}^{-1}\psi \left( \omega
_{N}^{-1}t^{-N}\right) ^{-1}$ (where, as usual, $\psi ^{-1}=1/\psi ,$ not $%
\phi $). In the simple case when $\phi (\tau )=\tau ^{p}$ with $p\geq 1,$
this yields 
\begin{equation}
||(2(tH)^{*}+m_{g})^{-1}||_{p}=m_{g}^{-1}\omega _{N}^{1/p}t^{N/p}.
\label{46}
\end{equation}
Once again, $((tH)^{*}\check{)}=(tH)^{*}$ by Theorem \ref{th13} (i). Thus,
if $\check{g}^{-1}\in L^{p}$ (equivalent to $(2\check{g}-m_{g})^{-1}\in L^{p}
$), it follows from (\ref{43}) and (\ref{46}) that $u(t,\cdot )^{-1}\in L^{p}
$ and that $||u(t,\cdot )^{-1}||_{p}=O(t^{N/p})$ as $t\rightarrow \infty $
(which is sharp because of (\ref{42})). This is the general estimate in
Theorem \ref{th23} (i) which, in this example, is therefore optimal.
\end{example}

In Example \ref{ex4}, $g^{-1}\in L^{p}$ is not enough to get $u(t,\cdot %
)^{-1}\in L^{p}:$ If $N=1$ and if $g=f^{-1/p}$ with $p<\infty $ and $f$ from
Example \ref{ex1}, then $m_{g}=1,$ $g^{-1}=f^{1/p}\in L^{p},$ but $%
u(t,x)=\inf_{|y|\leq t}f^{-1/p}(x-y)=1$ if $t\geq 1,$ so that $u(t,\cdot %
)^{-1}=1\notin $ $L^{p}.$

\section{Solutions by the Hopf formula\label{hopf}}

The Hopf formula (Hopf \cite{Ho65}) 
\begin{equation*}
u(t,\cdot )=(tH+g^{*})^{*}
\end{equation*}
gives a solution of (\ref{39}) when $g$ is convex and various other
technical assumptions are satisfied. See for instance Bardi and Evans \cite
{BaEv84}. In Penot and Volle \cite{PeVo00}, $g$ and $H$ can be extended
real-valued. Below, we assume that $g$ is lsc. Since $u(t,\cdot )$ is
convex, there is no measurability issue.

First, $-g^{*}(0)=\inf g=m_{g}$ (finite, as assumed above) by definition of $%
g^{*}.$ Next, if $h,k:\Bbb{R}^{N}\rightarrow (-\infty ,\infty ]$ and $h\not%
{\equiv}\infty ,k\not{\equiv}\infty $ (so that $h^{*}$ and $k^{*}$ are
proper), it is well-known and easily checked that $(h+k)^{*}\leq h^{*}\Box
k^{*}.$ As a result, $u(t,\cdot )=(tH+g^{*})^{*}\leq (tH)^{*}\Box
g^{**}=(tH)^{*}\Box g.$ On the other hand, by using ``$\inf \sup \geq \sup
\inf $'', we get $\inf_{x}(tH+g^{*})^{*}(x)\geq \sup_{y}\inf_{x}(x\cdot %
y-(tH+g^{*})(y))=-tH(0)-g^{*}(0)=-tH(0)+m_{g}.$ Indeed, $\inf_{x}(x\cdot %
y-(tH+g^{*})(y))=-\infty $ if $y\neq 0$ because $tH$ and $g^{*}$ are proper.
This shows that if $-tH(0)+m_{g}\geq 0,$ then $0\leq u(t,\cdot )\leq %
(tH)^{*}\Box g.$ Furthermore, since $(tH)^{*}=(tH_{C})^{*}$ where $%
H_{C}=H^{**}$ is the closed convex hull of $H,$ it follows that $%
m_{(tH)^{*}}=-tH_{C}(0)\geq -tH(0).$ To ensure that $m_{(tH)^{*}}\in \Bbb{R},
$ i.e., that $H_{C}(0)>-\infty ,$ it must be assumed that $H$ is bounded
below by an affine function. If so, it follows from Theorem \ref{th7} with $%
f=(tH)^{*}$ that if $-tH(0)+m_{g}\geq 0$ (hence $-tH_{C}(0)+m_{g}\geq 0$),
then 
\begin{multline}
||(2(tH)^{*}+tH_{C}(0)+m_{g})^{-1}||_{\phi
}+||(2g-tH_{C}(0)-m_{g})^{-1}||_{\phi }\leq   \label{47} \\
2||u(t,\cdot )^{-1}||_{\phi },
\end{multline}
for every Young function $\phi .$

Since the inequality (\ref{47}) depends only upon $H_{C},$ it remains true
when $H$ is replaced with any proper closed convex function $K\leq H$ (so
that $K=K_{C}$) under the same assumption $-tH(0)+m_{g}\geq 0$ as above
(still needed to ensure $u\geq 0$) because this substitution decreases the
left-hand side. This is less accurate, but often more convenient for
practical evaluation.

As an illustration of this point, it follows from the discussion in Example 
\ref{ex3} that $u(t,\cdot )^{-1}\notin L^{p}$ if $H(x)\geq |x|^{2}/2$ and
either $0<p\leq N/2$ or $m_{g}=0.$ Alternatively, from Example \ref{ex4}, $%
u(t,\cdot )^{-1}\notin L_{\phi }$ for any $\phi $ if $H(x)\geq |x|$ and $%
m_{g}=0.$

In the opposite direction, if $L\geq H$ is any lsc convex function, then $%
(tH+g^{*})\leq (tL+g^{*})$ and, since both $tL$ and $g^{*}$ are proper and
lsc, $(tL+g^{*})^{*}=(tL)^{*}\Box g$ as soon as the relative interiors of $%
\limfunc{dom}L$ and $\limfunc{dom}g^{*}$ have nonempty intersection (%
\cite[p. 145]{R070}). If so, $(tL)^{*}\Box g\leq u(t,\cdot ),$ so that the
reverse inequalities of Theorem \ref{th19} can be used with $f=(tL)^{*}$ if $%
-tL(0)+m_{g}\geq 0.$ Since $0\leq -tL(0)+m_{g}=m_{(tL)^{*}\Box g}\leq
u(t,\cdot ),$ it follows that 
\begin{multline}
||u(t,\cdot )^{-1}||_{\phi }\leq  \label{48} \\
2^{N-1}\left( ||(2((tL)^{*}\check{)}+tL(0)+m_{g})^{-1}||_{\phi }+||(2\check{g%
}-tL(0)-m_{g})^{-1}||_{\phi }\right) ,
\end{multline}
for every Young function $\phi .$ Unlike (\ref{47}), the inequality (\ref{48}%
) does not require $H$ to be bounded below by an affine function.

For instance, by Example \ref{ex3}, $||u(t,\cdot )^{-1}||_{p}=O(t^{N/2p})$
as $t\rightarrow \infty $ if $H(x)\leq L(x):=|x|^{2}/2,m_{g}>0$ and $\check{g%
}^{-1}\in L^{p},p>N/2.$ Also, by Example \ref{ex4}, $||u(t,\cdot
)^{-1}||_{p}=O(t^{N/p})$ as $t\rightarrow \infty $ if $H(x)\leq
L(x):=|x|,m_{g}>0$ and $\check{g}^{-1}\in L^{p},p\geq 1.$

\subsection{Explicit solutions of related problems\label{other}}

Various explicit formulas for the solution of 
\begin{equation}
\left\{ 
\begin{array}{l}
u_{t}+H(u,\nabla u)=0\text{ on }(0,\infty )\times \Bbb{R}^{N}, \\ 
u(0,\cdot )=g\text{ on }\Bbb{R}^{N},
\end{array}
\right.  \label{49}
\end{equation}
when the Hamiltonian depends upon $u,$ have been obtained under suitable
(but restrictive) conditions. The result most directly relevant to this
paper can be found in the work of Barron et al. \cite{BaJeLi96},
complemented and generalized in \cite{AlBaIs99}. It is shown in 
\cite[Theorem 6.11]{AlBaIs99} that if $H$ is continuous on $\Bbb{R}^{N+1},$
with $H(s,x)$ nondecreasing in $s\in \Bbb{R},$ convex and positively
homogeneous of degree $1$ in $x\in \Bbb{R}^{N},$ and if $g:\Bbb{R}%
^{N}\rightarrow (-\infty ,\infty ]$ is lsc, then (for $t>0$) 
\begin{equation}
u(t,\cdot ):=h_{[t]}\veebar g,  \label{50}
\end{equation}
is the minimal lsc supersolution of (\ref{49}), where\footnote{%
This differs from $h_{t}$ previously defined by $th(t^{-1}\cdot ).$} $%
h_{[t]}(x):=h\left( t^{-1}x\right) $ and 
\begin{equation}
h(x):=\inf \{s\in \Bbb{R}:H(s,\cdot )^{*}(x)\leq 0\}.  \label{51}
\end{equation}
It is shown in \cite{BaJeLi96} that $h$ is quasiconvex and lsc and that $%
m_{h}=-\infty .$

From now on, we assume $g\geq 0,$ so that $u(t,\cdot )=h_{[t]+}\veebar
g(=h_{[t]}\veebar g).$ Note that $h_{[t]+}=h_{+[t]}.$ Since $m_{h}=-\infty $
implies $m_{h_{[t]}}=-\infty ,$ it follows that $m_{h_{[t]+}}=0,$ so that $%
m_{h_{[t]+},g}=-m_{g}/2\leq 0.$

Corollary \ref{cor8} and Corollary \ref{cor20} with $f=h_{[t]}$ can be used
to evaluate $||u(t,\cdot )^{-1}||_{\phi }.$ Specifically, 
\begin{equation}
||(2h_{[t]+}+m_{g})^{-1}||_{\phi }+||(2g-m_{g})^{-1}||_{\phi }\leq
2||u(t,\cdot )^{-1}||_{\phi }  \label{52}
\end{equation}
and 
\begin{equation}
||u(t,\cdot )^{-1}||_{\phi }\leq 2^{N+1}\left( ||(2(h_{[t]+}\check{)}
+m_{g})^{-1}||_{\phi }+||(2\check{g}-m_{g})^{-1}||_{\phi }\right) ,
\label{53}
\end{equation}
for every Young function $\phi .$

\begin{example}
\label{ex5}\textbf{\ }Let $H(s,x)=(s_{+})^{\alpha }|x|$ with $\alpha >0.$
When $g\geq 0,$ the (nonnegative) solution (\ref{50}) of (\ref{49}) actually
solves $u_{t}+u^{\alpha }|\nabla u|=0.$ By (\ref{51}) and a straightforward
calculation, $h(0)=-\infty $ and $h(x)=|x|^{1/\alpha }$ if $x\neq 0.$\newline
\quad If $m_{g}=0,$ then $u(t,\cdot )^{-1}\notin L^{p}$ for any $p\geq 1$ by
(\ref{52}) since $h_{[t]+}^{-1}=t^{1/\alpha }|\cdot |^{-1/\alpha }\notin
L^{p}.$ However, it is readily checked that $h_{[t]+}^{-1}\in L^{1}+L^{p}$
if $1<N\alpha <p.$ If so, it follows from (\ref{53}) with $m_{g}=0$ that $%
u(t,\cdot )\in L^{1}+L^{p}$ if $\check{g}^{-1}\in L^{1}+L^{p}$ (hence $%
g^{-1}\in L^{1}+L^{p}$ by (\ref{52})) and the calculation of the estimates (%
\ref{52}) and (\ref{53}) is trivial since $m_{g}=0$ and since $%
h_{[t]+}=t^{-1/\alpha }|\cdot |^{1/\alpha }=(h_{[t]+}\check{)}$ by Theorem 
\ref{th13} (i). Thus, $||(h_{[t]+})^{-1}||_{L^{1}+L^{p}}=||((h_{[t]+}\check{)%
})^{-1}||_{L^{1}+L^{p}}=Ct^{1/\alpha }$ with $C:=||\,|\cdot |^{-1/\alpha
}\,||_{L^{1}+L^{p}},$ which yields $||u(t,\cdot
)^{-1}||_{L^{1}+L^{p}}=O(t^{1/\alpha })$ as $t\rightarrow \infty $ (optimal).%
\newline
\quad If $m_{g}>0,$ then $(2h_{[t]+}+m_{g})^{-1}\in L^{p}$ if and only if $%
p>N\alpha .$ If so and if $\check{g}^{-1}\in L^{p}$ (equivalent to $(2\check{%
g}-m_{g})^{-1}\in L^{p}$), it follows from (\ref{53}) that $u(t,\cdot
)^{-1}\in L^{p}$ with $||u(t,\cdot )^{-1}||_{p}=O(t^{N/p})$ as $t\rightarrow
\infty $ (optimal by (\ref{52})).
\end{example}

Example \ref{ex4} is recovered when $\alpha =0$ in Example \ref{ex5}. If so, 
$h(x)=-\infty $ on the closed unit ball and $h(x)=\infty $ outside, so that $%
h_{+}$ is the indicator function of the closed unit ball. Even though the
formulas (\ref{40}) and (\ref{50}) look different, they both provide the
same solution $u(t,x)=\inf_{|y|\leq t}g(x-y).$

\begin{example}
\label{ex6}Let $H(s,x)=e^{s}|x|,$ so that, by (\ref{51}), $h(0)=-\infty $
and $h(x)=\ln |x|$ if $x\neq 0.$ Thus, $h_{+}(x)=\ln _{+}|x|$ is continuous,
radially symmetric and nondecreasing in $|x|,$ so that, once again, $%
(h_{[t]+}\check{)}=h_{[t]+}$ by Theorem \ref{th13} (i). As always in this
subsection, $g\geq 0.$\newline
Since $h_{[t]+}=0$ on the ball with center $0$ and radius $t,$ it follows
from (\ref{52}) that $u(t,\cdot )^{-1}\notin L_{\phi }$ for any $t>0$ and
any Young function $\phi $ if $m_{g}=0.$ In addition, the slow growth of $%
\ln |x|$ reveals that $u(t,\cdot )^{-1}\notin L^{p},$ $1\leq p<\infty ,$
even if $m_{g}>0.$\newline
However, if $m_{g}>0,$ then $(2h_{[t]+}+m_{g})^{-1}\in L_{\phi }$ if $\phi $
vanishes fast enough at the origin. Aside from $L^{\infty },$ one of the
simplest examples is given by the Young function $\phi (\tau ):=e^{\tau
-\tau ^{-2}}.$ (The growth of $\phi $ at infinity could be damped
considerably to enlarge the space $L_{\phi }.$) Thus, if $m_{g}>0$ and $%
\check{g}^{-1}\in L_{\phi }$ (equivalent to $(2\check{g}-m_{g})^{-1}\in
L_{\phi }$), it follows from (\ref{53}) that $u(t,\cdot )^{-1}\in L_{\phi }$
for every $t>0.$ We did not attempt to estimate $||u(t,\cdot )^{-1}||_{\phi
} $ as $t\rightarrow \infty .$
\end{example}

\end{document}